\documentclass[12pt, reqno, a4paper]{amsart}

\usepackage{ amssymb, amsmath, enumerate, amsfonts, amsthm, mathrsfs, url, bm, mathtools}

\numberwithin{equation}{section}

\usepackage{xcolor}  	
\usepackage[backref]{hyperref}
\hypersetup{
	colorlinks,
    linkcolor={blue!60!black},
    citecolor={blue!60!black},
    urlcolor={red!60!black}
}

\usepackage{color}

\usepackage[normalem]{ulem}

\usepackage[makeroom]{cancel}

\usepackage[margin=1.25in]{geometry}

\RequirePackage{doi}

\usepackage[square,sort,comma,numbers]{natbib}
\makeatletter
\@namedef{subjclassname@2020}{%
  \textup{2020} Mathematics Subject Classification}
\makeatother

 \newcommand{\set}[1]{\left\{#1\right\}}

\newcommand{\Bigset}[1]{\Bigl\{ #1 \Bigr\}}

\newcommand{\bigabs}[1]{\bigl| #1 \bigr|}
\newcommand{\Bigabs}[1]{\Bigl| #1 \Bigr|}

\newcommand{\floor}[1]{\left\lfloor #1 \right\rfloor}

\newcommand{\bigbrac}[1]{\bigl( #1 \bigr)}
\newcommand{\Bigbrac}[1]{\Bigl( #1 \Bigr)}

\newcommand{\norm}[1]{\left\| #1\right\|}

\newcommand{\rd}{\,\mathrm{d}}

\newcommand{\twosum}[2]{ \sum_{\substack{#1\\ #2}}}

\newcommand{\Z}{\mathbb{Z}}

\newcommand{\R}{\mathbb{R}}
\newcommand{\C}{\mathbb{C}}
\newcommand{\T}{\mathbb{T}}

\newcommand{\supp}{\mathrm{supp}}

\newcommand{\eps}{\varepsilon}

\makeatletter
\let\@@pmod\pmod
\DeclareRobustCommand{\pmod}{\@ifstar\@pmods\@@pmod}
\def\@pmods#1{\mkern4mu({\operator@font mod}\mkern 6mu#1)}
\makeatother

\newcommand{\major}{\mathfrak M}
\newcommand{\minor}{\mathfrak m}

\newcommand{\mrd}[1]{\,(\mathrm{mod}\,#1)}
\renewcommand{\bar}{\overline}
\renewcommand{\hat}{\widehat}

\renewcommand{\leq}{\leqslant}
\renewcommand{\geq}{\geqslant}
\renewcommand{\epsilon}{\varepsilon}
\newcommand{\mmod}[1]{\,\,(\text{\rm mod}\,\, #1)}

\newtheorem{theorem}{Theorem}[section]

\newtheorem{conjecture}[theorem]{Conjecture}

\newtheorem{proposition}[theorem]{Proposition}
\newtheorem{lemma}[theorem]{Lemma}

\theoremstyle{definition}

\newtheorem*{remark}{Remark}

\numberwithin{theorem}{section}

\begin{document}
\title[Local divisor correlations in almost all short intervals]{Local divisor correlations in almost all short intervals}

\author{Javier Pliego}

\address{Dipartimento di Matematica, Universit\`a degli Studi di Genova,  Via Dodecaneso 35, 16146 Genova, Italy}

\email{javier.pliego.garcia@edu.unige.it}

\author{Yu-Chen Sun}

\address{School of Mathematics, University of Bristol, Bristol, BS8 1UG, England}

\email{yuchensun93@163.com}

\author{Mengdi Wang}

\address{\'{E}cole Polytechnique F\'{e}d\'{e}rale de Lausanne (EPFL), Lausanne, Switzerland}
\email{mengdi.wang@epfl.ch}

\maketitle

\begin{abstract}

Let $ k,l \geq 2$ be natural numbers,  and let $d_k,d_l$ denote the $k$-fold and $l$-fold divisor functions, respectively. We analyse the asymptotic behavior of the sum $\sum_{x<n\leq  x+H_1}d_k(n)d_l(n+h)$. More precisely, let $\eps>0$ be a small fixed number and let $\Phi(x)$ be a positive function that tends to infinity arbitrarily slowly as $x\to \infty$. We then show that whenever $H_1\geq(\log x)^{\Phi(x)}$ and $(\log x)^{1000k\log k}\leq H_2\leq H_1^{1-\eps }$,  the expected asymptotic formula holds for almost all $x\in[X,2X]$ and almost all $1\leq h\leq H_2$.

\end{abstract}

\section{Introduction}

Let $k\geq 2$ be a natural number. Define the $k$-fold divisor function $d_k(n)$ as the number of ways that $n$ can be expressed as the product of $k$ natural numbers. When $k$  is moderately large, the behavior of $d_k$ exhibits in many ways a resemblance to the von Mangoldt function, Heath-Brown's identity \cite{HB} providing an explicit connection between both objects. Consequently, there is an analogous version of the \emph{Hardy-Littlewood prime tuple conjecture} in the context of divisor functions. We record herein the statement from \cite[Conjecture 3]{CG} for the sake of transparency.

\begin{conjecture}[Divisor correlation conjecture]
	Suppose that $k,l\geq 2$ are natural numbers, and $ h\in\Z\setminus\{0\}$. Then one has
	\[
	\sum_{X<n\leq 2X} d_k(n) d_l(n+h) =P_{k,l,h} (\log X) X +O(X^{1/2+o(1)})
	\]
	as $X\to\infty$ for some polynomial $P_{k,l,h}$ of degree $k+l-2$.
\end{conjecture}

This conjecture remains wide open for $k,l\geq3$. Nonetheless, some progress has been made for the case $l=2$, albeit with a weaker error term. For instance, \cite{DI} established the above formula when $k=l=2$ with an error term of the shape $O(X^{2/3+o(1)})$, while \cite{Dr, To} demonstrated for $k\geq3$ that there is some absolute constant $\delta_k>0$ such that
\[
\sum_{X<n\leq 2X}d_k(n) d_2(n+h) = P_{k,2,h} (\log X) X + O(X^{1-\delta_k}).
\]
We refer interested readers to \cite{Dr,MRTI} for  comprehensive overviews on this topic. 

A weaker version of this conjecture considers obtaining asymptotics for almost all $h$ in a given region of the shape $[h_0-H_2,h_0+H_2]$ but with weak error terms, the problem becoming closer to the original conjecture as $H_2$ gets smaller. In this direction, \cite{MRTI} established 
\begin{align*}
\sum_{X<n\leq 2X} d_k(n) d_l(n+h) = P_{k,l,h} (\log X) X +O(X\log^{-A}X)	
\end{align*}
for almost all $h\leq H_2$, provided that $X^{8/33-\eps}\leq H_2\leq X^{1-\eps}$. Prior to this,  \cite{BBMZ} had shown the above with the corresponding error term exhibiting a power-saving bound when $k=l=3$ and $H\geq X^{1/3+\eps}$. Furthermore, the first authors \cite{MRTII} achieved an analogous conclusion for the range $H_{2}\geq (\log X)^{10000k\log k}$ at the cost of enabling weaker error terms of the form $o(X(\log X)^{k+l-2})$. 


We shall be concerned herein with the problem of restricting the supporting interval of summation to a smaller scale, which aids in better understanding the local behavior of divisor correlations. We should first observe that a direct application of \cite[Theorem 1.1]{MRTII} (see \cite[Lemma 2.4]{Wang} for details) reveals that whenever $H_1\geq H_2\geq (\log X)^{10000k\log k}$ then the upper bound 
\[
\sum_{x<n\leq x+H_1} d_k(n) d_l(n+h) \ll_{k,l,h} H_1 (\log X)^{k+l-2}
\]
holds for almost all $x\in[X,2X]$ and almost all $h\leq H_2$. Moreover, experts may notice that the circle method combined with \cite[Theorem 1.1]{Wang} permits one to achieve an asymptotic formula for the range $H_1\geq H_2\geq \exp(C(\log X)^{1/2} (\log \log X)^{1/2})$. In this paper, we aim to address the question of establishing the expected asymptotic formula for smaller values of $H_1$ and $H_2$.

\begin{theorem}\label{thm1.1}
	Let $k,l\geq 2$ be natural numbers and $0<\eps_{1},\eps_{2}<1$ be small constants. Suppose that $\Phi(X) \to\infty$ arbitrarily slowly as $X\to \infty$. Then, whenever $(\log X)^{\Phi(X)}\leq H_1\leq X^{1-\eps_{1} }$ and $(\log X)^{1000k\log k}\leq H_2\leq H_1^{1-\varepsilon_{2}}$ one has
	\[
	\sum_{x\sim X}\sum_{0<|h|\leq H_2} \Bigabs{\sum_{x<n\leq x+H_1} d_k(n) d_l(n+h)-c_{k,l,h}(\log X)^{k+l-2} H_1}=o(XH_1H_2\log^{k+l-2}X),
	\]
	where $c_{k,l,h}>0$ is some explicit constant defined in (\ref{chk}).
\end{theorem}

The parameter $H_2$ in Theorem \ref{thm1.1} is optimal up to a power of $\log X$, since when $H_2< (\log X)^{k\log k-k+1-\eps }$,  the main contribution to the average $\sum_{|h|\leq H_2}d_k(X+h)$ arises from integers with $(1+o(1))k \log\log X$ prime factors, as shown in \cite[Theorem 1.2]{Sun}. The restriction on $H_1$ stems from the limitation of the use of the circle method and Matom\"aki-Radziwi{\l}{\l} theorem.


\subsection{Organization of the paper}

We first reduce the question to typical numbers $n$ by introducing the function $f_k$ in (\ref{f_k}) in Section \ref{sec2}. Then, in Section \ref{sec3}, we establish a comparison result for short and long averages in terms of the divisor function which serves as a preparation for both major-arc and minor-arc estimates. Our result exhibits a stronger error term to that in the classical Matom\"aki-Radziwi{\l }{\l } theorem, as it benefits from the use of the Vinogradov-Korobov zero-free region in the context of divisor functions. We present this result in a broader form, anticipating its applicability in other contexts. Since our framework relies on the circle method, we establish the major-arc asymptotic in Section \ref{sec4}. As far as we are concerned, this is the first attempt to address an additive question with varying support ranges. In order to achieve so,  we adapt the long-short interval comparison argument for every major-arc frequency, which  transfers the analysis from short exponential sums to long ones, where explicit evaluation is possible. The restriction on the parameter $H_1$ in Theorem \ref{thm1.1} is imposed herein to ensure that the error terms when performing the comparison remain acceptable on average.

The minor-arc estimate in Section \ref{sec5} is inspired by that in \cite[Section 5]{MRTII}, it being pertinent remarking that a generalization of \cite[Lemma 5.1]{MRTII} is not straightforward in our setting. To illustrate the underlying ideas in the minor arc analysis, we anticipate that the goal is to prove
\[
\int_X^{2X} \int_{\minor\cap [\alpha-\frac{1}{2H_2}, \alpha+\frac{1}{2H_2}]} \Bigabs{\sum_{x<n \leq x+H_1} f_k(n) e(n\gamma)}^2\rd \gamma \rd x =o(XH_1 (\log X)^{2k-2}),
\] 
where $\minor\subset\T$ denotes the minor arcs,  $f_k:[X,2X]\to\R_{\geq0}$ is defined in (\ref{f_k}) and $H_2\leq H_1^{1-\eps  }\leq X^{1-\eps  }$. We note first that by extending the integral over $\gamma$ to the entire circle, an application of Plancherel's identity then yields the trivial bound $O(H_1\sum_{X/2<n<3 X} |f_k(n)|^2)$. This motivates us to transfer the analysis from the dual space to the physical space, ensuring that the summation intervals can be independent of the parameter $H_1$. A relevant technical result in this context is that in \cite[Proposition 5.1 (i)]{MRTI}, which asserts that
\begin{multline}\label{aa}
\int_{\beta-1/H}^{\beta+1/H} \Bigabs{ \sum_{n \sim X} f(n)  e(n\gamma)}^2\rd \gamma 
\ll \int_{|\beta|X(\log X)^{-A}}^{|\beta|X(\log X)^{B}} \Bigabs{\sum_{n\sim X}\frac{f(n)}{n^{1/2+it}}}^2\rd t \\
+
\bigbrac{(\log X)^{-10}+\frac{1}{|\beta|H}}^{2}	\int_X^{2X} \bigbrac{H^{-1}\sum_{x<n\leq x+H}|f(n)|}^2\rd x.
\end{multline}
This inequality can be used satisfactorily to handle the case when $\gamma\in\minor$ is highly irrational (i.e.$\norm{q \gamma}>H_2^{-1}$ or $|\beta|>H^{-1}$ in (\ref{aa})). However, this approach itself cannot address the estimate for frequencies lying on intermediate ranges (i.e. $H_1^{-1}<\norm{q\gamma}\leq H_2^{-1}$), which remain a challenge in our setting. In order to illustrate the discussion, we sketch our strategy to overcome this problem and assume for simplicity that $q=1$. We decompose these arcs into short subintervals $I_i=[\gamma_i-\eta_i\gamma_i,\gamma_i+\eta_i\gamma_i]$, where $\gamma_i$ is the central frequency of $I_i$ and $\eta_i$ is a parameter tending to zero. Applying (\ref{aa}) to each $I_i$ ensures that the second term remains negligible. Meanwhile, for arcs $I_i$ corresponding to medium frequencies, one can obtain a sharper upper bound for the first term in (\ref{aa}) which offsets the loss incurred by decomposing into shorter arcs. We believe that this approach extends the applicability of  Matom\"aki-Radziwi{\l }{\l } construction. In particular, our method provides an alternative proof of the minor-arc estimate in \cite{MRTII}.

In Section \ref{sec6}, we derive the specific mean value estimates, which are essentially a generalization of \cite[Proposition 3.6]{MRTII} and are referred to as the large values estimates here. Finally, we combine all the conclusions together in Section \ref{sec7} to deliver the main result of the paper.

\subsection*{Acknowledgement}

We thank Kaisa Matom\"aki for numerous valuable discussions and for inviting the first author to Turku for a week. JP was partially funded by the Curiosity Driven grant ``Value distribution of quantum modular forms'' of the Universita degli Studi di Genova. YS was supported by UTUGS funding, working in the Academy of Finland project no. 333707. MW was partially supported by the Academy of Finland grant no. 333707 and the Swiss National Science Foundation grant TMSGI2-2112.


\section{Preliminaries and technical reduction}\label{sec2}


Let $k\in\Z$ and $X$ be sufficiently large. Suppose that $\Psi(X)\to 0$ very slowly with $X\to\infty$  and $\Phi(X)\to \infty$ is the function defined in Theorem \ref{thm1.1}. Let $\eps'>0$ be a small fixed constant. We then denote $S_k$ as the set of integers $n\in[X,2X]$ that satisfy the two conditions below.
\begin{itemize}
	\item 	
	Let 
\begin{alignat}{3}\label{pq-1}
P_{1^{(1)}} &:= (\log X)^{\Psi(X)}, & \quad& 	Q_{1^{(1)}} := (\log X)^{80k\log k}, \nonumber\\
P_{1^{(2)}} &:= (\log X)^{10000k\log k}, & \quad& Q_{1^{(2)}} :=(\log X)^{0.1\Phi(X)}.
\end{alignat} 
 We also define 
\begin{alignat}{2}\label{pq-23}
P_2 &:= \exp\bigbrac{(\log\log X)^{2}}, & \quad& Q_3 := \exp \bigbrac{(\log\log X)^{5/2}},  \nonumber\\
P_3 &:= \exp\bigbrac{(\log X)^{3/4}}, & \quad& Q_3 := \exp \bigbrac{(\log X)^{5/6}}.
\end{alignat}

Write $n$ as 
 \[
 n=n_{1^{(1)}}n_{1^{(2)}}n_2n_3 m,
 \]
 where  the prime factors of each $n_{1^{(i)}}$ $ (1\leq i\leq 3)$  lie within the interval $[P_{1^{(i)}},Q_{1^{(i)}}]$; while the prime factors of every $ n_j $ $(j=2,3)$ lie within the interval $[P_j,Q_j]$;  and the remaining factor $m$ contains no prime factors within any of these intervals.  We then require that $n_{1^{(i)}} >1$ with $1\leq i\leq 3$ and $ n_j>1$  for $j=2,3$. 

 \item The total number of prime factors of $n$ does not exceed $(1+\eps' )k\log \log X$, and the total number of prime factors of $n$ in the range $[X^{\frac{1}{(\log\log X)^2}}, 2X]$ does not exceed $10k\log\log\log X$.
\end{itemize}

Now we set
\begin{align}\label{f_k}
f_k(n)=d_k(n)1_{S_k}(n).
\end{align}

\begin{remark}
The above construction of $S_k$ actually constitutes a modification of the set $S_{k,X}$ in the paper  \cite[pages 806-807]{MRTII}. In particular, the function $\Psi(X)$ plays the role of balancing the minor arc and large value estimates, as will become apparent in Section \ref{sec7}.

 The two conditions in the second assumption serve different purposes. The first one allows our local correlation parameter $H_2$ to be  of the shape  $O(\log^{Ck\log k}X)$ rather than $O(\log^{Ck^2}X)$, where $C>0$ is an absolute constant. The second condition is employed to construct a majorant function for $d_k$. This majorant function is essential to establish mean value estimates without logarithmic power loss, as discussed in \cite[Lemma 6.1]{MRTII}.

\end{remark}

The lemma below partially illustrates that in divisor correlation problems, no information is lost when restricting $d_k$ to the set $S_k$. We shall henceforth write $x\sim X$ to denote $X< x\leq 2X$.
\begin{lemma} \label{contribution-dk-fk}
Suppose that $k,l\geq 2$ and $\log^{1000k\log k}X\leq H_{2}\leq H_1^{1-\eps_{1}}$ with $H_{1}\leq X^{1-\eps_2 } $. Let $f_k:\Z\to\R_{\geq0}$ be the function defined in (\ref{f_k}). One then has
\[
\sum_{x\sim X} \sum_{0<|h|\leq H_2}\sum_{x<n\leq x+H_1} |d_k(n)-f_k(n)|d_l(n+h)=o(XH_1H_2 \log^{k+l-2}X).
\]	
\end{lemma}

\begin{proof}
	By switching the order of summations one may find that
	\begin{multline*}
		\sum_{x\sim X} \sum_{0<|h|\leq H_2}\sum_{x<n\leq x+H_1} |d_k(n)-f_k(n)|d_l(n+h)\\
		=\sum_{X<n\leq 2X+H_1}\sum_{0<|h|\leq H_2} |d_k(n)-f_k(n)|d_l(n+h)\sum_{n-H_1\leq x<n}1,
	\end{multline*}
	the latter term being bounded by $o(XH_1H_2\log^{k+l-2}X)$ in light of \cite[Lemma 3.3]{MRTII}.
	\end{proof}
	
We then claim that in order to prove Theorem \ref{thm1.1}, it suffices to demonstrate the following proposition.
\begin{proposition}\label{reduction}
	Suppose that $k,l\geq 2$ and $f_k,f_l:[X,2X]\to\R_{\geq0}$ are functions defined as in (\ref{f_k}). Then when $ (\log X)^{\Phi(X)}\leq H_1\leq X^{1-\eps_1 }$ and $\log^{1000k\log k}X\leq H_2\leq H_1^{1-\eps_2}$ one has
	\[
	\sum_{x\sim X}\sum_{0<|h|\leq H_2} \bigabs{\sum_{x<n\leq x+H_1} f_k(n) f_l(n+h)-  c_{h,k,l} H_1  \log^{k+l-2}X }=o(XH_1H_2\log^{k+l-2}X).
	\]
\end{proposition}	

Indeed, upon noting that $f_k\leq d_k$ pointwise, it transpires that the employment of Lemma \ref{contribution-dk-fk} in conjunction with Proposition \ref{reduction}, a symmetry argument and the triangle inequality entails
\begin{multline*}
	\sum_{x\sim X}\sum_{0<|h|\leq H_2} \bigabs{\sum_{x<n\leq x+H_1} d_k(n) d_l(n+h)- c_{h,k,l} H_1 \log^{k+l-2}X}\\
	\ll \sum_{x\sim X}\sum_{0<|h|\leq H_2} \bigabs{\sum_{x<n\leq x+H_1} f_k(n) f_l(n+h)- c_{h,k,l} H_1 \log^{k+l-2}X}\\
	 + \sum_{x\sim X} \sum_{0<|h|\leq H_2}\sum_{x<n\leq x+H_1} |d_k(n)-f_k(n)|d_l(n+h)=o(XH_1H_2\log^{k+l-2} X).
\end{multline*}


\section{Distribution of divisor function twisted by Dirichelet characters in almost all very short intervals}\label{sec3}

We start this section with a lemma that is crucial for building our  Matom{\"a}ki-Radziwi{\l}{\l} type conclusion for divisor functions. 

\begin{lemma}[Dirichlet polynomial of primes] \label{dirichlet-prime}
Suppose that $2/3<\theta<1$ and $\exp((\log X)^\theta)\leq P\leq Q\leq X$. Suppose that $t\in\R $ with $|t|\leq X^{1+\eps }$ and $\chi$ is a Dirichlet character with conductor $q\leq \log^{100k\log k}X$. Then there is some positive constant $c>0$ such that
\[
\Bigabs{\sum_{P\leq p\leq Q} \frac{\chi(p)}{p^{1+it}}} \ll\frac{\log X}{1+|t|} + \exp\Bigbrac{-\frac{c(\log X)^{\theta-2/3}}{(\log\log X)^{1/3}}}.
\]	
Moreover, when $\chi$ is not the principal character one has
\[
\Bigabs{\sum_{P\leq p\leq Q} \frac{\chi(p)}{p^{1+it}}} \ll \exp\Bigbrac{-\frac{c(\log X)^{\theta-2/3}}{(\log\log X)^{1/3}}}.
\]
\end{lemma}

\begin{proof}
We first introduce the so called Vinogradov-Korobov zero-free region. Suppose first that $|\Im s+t|\geq 10$. By \cite[Theorem 1.1]{Kh} and the assumption $q\leq (\log X)^{100k\log k}$, one may deduce that  $ L(s+1+it,\chi)$ does not vanish in the region
\[
\Re s\geq -\frac{c_1}{(\log X)^{2/3} (\log\log X)^{1/3}} \quad 10\leq |\Im s+t|\leq X^{1+\eps},
\]
for some constant $c_1>0$. On the other hand, if $|\Im s+t|\leq 10$ then the application of \cite[Corollaries 11.8 and 11.15]{MV}  shows that for any $\eps>0$ there is another constant $c_2=c_2(\eps)>0$ such that $ L(s+1+it,\chi)\neq 0$ whenever $\Re s\geq-c_2/q^\eps.$
 Therefore, combining the preceding cases, we can always conclude that $L(s+1+it,\chi) \neq 0$ in the region
\begin{align*}
\mathcal R_1:   |\Im s+t|\leq X^{1+\eps }, \quad \Re s\geq-\sigma_0=-\frac{c'}{q^\eps+(\log X)^{2/3} (\log\log X)^{1/3}},
\end{align*}
where $c'>0$ is an absolute constant. If $\chi \neq \chi_0$, this implies that  $\log L(s+1+it,\chi)$ is well-defined in the region $\mathcal R_1$. Furthermore, by \cite[Theorem 11.4 (11.6)]{MV} the upper bound 
\begin{align}\label{log-bound}
|\log L(s+1+it,\chi)| \ll \log^2 X
\end{align}
holds for any character $\chi\neq\chi_0$. We aim to make further progress in the lemma by expressing the above Dirichlet polynomial as a contour integral involving the function $\log L(s+1+it,\chi)$. 
 
We first deal with the case when $\chi\neq\chi_0$. Without loss of generality, we may assume that the fractional part of both $P$ and $Q$ are $1/2$, and take a small constant $0<\kappa<1/2$. It then follows from the application of the truncated Perron's formula (e.g. \cite[Theorems 5.2 and 5.3]{MV}) that
\begin{multline*}
	\sum_{P\leq p\leq Q} \frac{\chi(p)}{p^{1+it}} = \frac{1}{2\pi i} \int_{\kappa-iX}^{\kappa+iX} \log L(s+1+it,\chi) \frac{Q^s-P^s}{s}\rd s+ O\bigbrac{\frac{Q^\kappa}{X} \sum_{p\geq 1} p^{-1-\kappa}}\\
	+O\Bigbrac{ \sum_{n\sim Q}n^{-1}\min\set{1,\frac{Q}{X|Q-n|}}} +O\Bigbrac{ \sum_{n\sim P} n^{-1}\min\set{1,\frac{P}{X|P-n|}}}.
\end{multline*}
The first error term in the above equation is $O(Q^\kappa/X)$. In order to deal with the second one, one may recall that $1/2\leq \lvert Q-n\rvert$ by the above assumption to deduce that such an error term is $O(X^{-1}\log Q )$, an analogous calculation entailing that the third one is $O(X^{-1}\log P )$. Therefore, we have
\[
\sum_{P\leq p\leq Q} \frac{\chi(p)}{p^{1+it}} = \frac{1}{2\pi i} \int_{\kappa-iX}^{\kappa+iX} \log L(s+1+it,\chi) \frac{Q^s-P^s}{s}\rd s +O(X^{-1/2}).
\]
In light of (\ref{log-bound}) and the assumption $\kappa<1/2$, we may shift the contour to the edge of the region $\mathcal R_1$ to obtain 
\begin{align*}
\sum_{P\leq p\leq Q}& \frac{\chi(p)}{p^{1+it}} = \frac{1}{2\pi i} \int_{-X}^{X} \log L(1-\sigma_0+i(t+u),\chi) \frac{Q^{-\sigma_0+iu} -P^{-\sigma_0+iu}}{-\sigma_0+iu}\rd u\\
& \pm \frac{1}{2\pi i} \int_{-\sigma_0}^\kappa \log L(1+\sigma \pm i(t+X),\chi) \frac{Q^{\sigma\pm iX} -P^{\sigma\pm iX}}{\sigma\pm iX}\rd\sigma +O(X^{-1/2}).	
\end{align*}
We then take absolute values inside the integral to conclude that
\[
\Bigabs{ \sum_{P\leq p\leq Q}\frac{\chi(p)}{p^{1+it}}} \ll P^{-\sigma_0}\log^3X+ X^{-1/2} \ll \exp\Bigbrac{-\frac{c(\log X)^{\theta-2/3}}{(\log\log X)^{1/3}}}
\]
recalling that $X\geq Q\geq P\geq \exp((\log X)^\theta)$ and $2/3<\theta\leq 1$, and taking $c=c'/2$.

Now let's briefly prove the estimate when $\chi=\chi_0$. We first note that the prime number theorem entails that the Dirichlet polynomial at hand is $O(\log\log Q-\log\log P)$ whenever $|t|\leq 20$. Thus, in what follows, we may always assume that $20< |t|\leq X^{1+\eps}$. 

Let $\delta= (\log X)^{-1}$  and  $T=|t|-10$. Then, an application of the truncated version of Perron's formula delivers
\begin{multline*}
	\sum_{P\leq p\leq Q} \frac{\chi_0(p)}{p^{1+it}} = \frac{1}{2\pi i} \int_{\delta-iT}^{\delta+iT} \log L(s+1+it,\chi_0) \frac{Q^s-P^s}{s}\rd s+ O\bigbrac{T^{-1} \sum_{p\geq 1} p^{-1-\delta }}\\
	+O\Bigbrac{T^{-1}\sum_{1/2< \lvert n-Q\rvert\leq Q}(n-Q)^{-1}+T^{-1}\sum_{1/2< \lvert n-P\rvert\leq P}(n-P)^{-1}}.
\end{multline*}
The same argument as above permits one to deduce that the preceding error term is $O(\log X/(|t|+1))$. Note that in the above integral $1\leq|\Im s+t|\leq 2 X^{1+\eps }$. Since $ L(s+1+it,\chi_0)$ has no poles under such assumptions, then $\log L(s+1+it,\chi_0)$ is well-defined in the region
\[
\mathcal R_2:   1\leq|\Im s+t|\leq 2X^{1+\eps }, \quad \Re s\geq-\sigma_0,
\]
and moreover $|\log L(s+1+it,\chi_0)|\ll \log^2X$ whenever $s\in\mathcal R_2$. Therefore, as shown in the case of non-principal characters, by shifting the contour to the edge of $\mathcal R_2$ we thereby conclude that
\[
\Bigabs{ \sum_{P\leq p\leq Q}\frac{\chi_0(p)}{p^{1+it}}} \ll \frac{\log X}{1+|t|} + \exp\bigbrac{-\frac{c(\log X)^{\theta-2/3}}{(\log\log X)^{1/3}}}.
\]
The lemma follows by combining the conclusions from the cases $|t|\leq 20$ and $|t|> 20$.
\end{proof}

The next result shows that when the range of summation is appropriately large, the average of the divisor function twisted by non-principal characters of small conductor has strongly logarithmic decay. This proposition will play a crucial role both when deriving our Matom\"aki-Radziwi{\l }{\l }-type conclusion for divisor functions and performing the major arc analysis.

\begin{proposition}\label{non-principal}
Suppose that $1\leq q, q_0\leq \log^{100k\log k}X$ are integers, $A>0$ is a fixed constant and $T\in\mathbb{R}$ with $\lvert T\rvert\leq X$. Let $f_k:[X,2X]\to\R_{\geq 0}$ be the function defined in (\ref{f_k}). Then, whenever $X\log^{-A}X\leq Y\leq X$ one has
\begin{enumerate}
	\item $\sum_{\chi\neq\chi_0 \mmod q} \bigabs{\sum_{X/q_0<n\leq X/q_0+Y}f_k(q_0 n)\chi(n)n^{iT} }\ll Y\log^{-K}X;$
	\item $\sum_{\chi\neq\chi_0 \mmod q} \bigabs{\sum_{X/q_0<n\leq X/q_0+Y}d_k(q_0 n)\chi(n)n^{iT} }\ll Y\log^{-K}X,$
\end{enumerate}
where $K>1$ is an arbitrarily large constant.	
\end{proposition}

\begin{proof}

The assertions follow immediately from the proof of \cite[Lemma 4.4]{MRTII}. 
\end{proof}

We have so far completed all the essential preliminary work, and we are now prepared to present the main conclusion of this section. Since this will be employed in different contexts and with various parameters, we aim to provide a fairly general result. For such purposes it seems worth introducing beforehand some notation. Suppose that $P_{1},Q_{1}$ are parameters satisfying $2<P_1<Q_1<P_2$, and recall $P_2,Q_2,P_3,Q_3$, these being defined in (\ref{pq-23}).

\begin{proposition}[Long-short averages]\label{long-short}

Let $1\leq q, q_0\leq \log^{100k\log k}X$ be integers and $T_{0}\in\R$ with $|T_{0}|\leq X$. Suppose that $g:[X,2X]\to\R_{\geq 0}$ is an arithmetic function that satisfies the following conditions.
\begin{itemize}
	\item $g(n)\leq d_k(n) 1_{\Omega(n)\leq k(1+\eps')\log\log X}$ pointwise;
	\item if $n\in \supp(g)$, then $(n,\prod_{P_j\leq p\leq Q_j} p)>1$ for all $1\leq j\leq 3$;
     \item For each $1\leq i\leq 3$ then whenever $p\in[P_i,Q_i]$  and $p\not|m$ with $mp\in \supp(g)$ one has $g(pm)=ka(m)$ for some arithmetic function $a:\mathbb{N}\rightarrow \mathbb{C}$ satisfying $|a(m)|\leq d_k(m)$.
\end{itemize}
   Let $X/\log ^{B}X\leq Y\leq X/\log^A X$ for sufficiently large constants $0<A<B$ and $\log^{1000k\log k}X\leq H\leq X^{1-\eps }$. Whenever $\log^{80k\log k} X\leq Q_1\leq \min\set{H^{2/3}, P_2^{1/100}}$ and $P_1\leq \min\set{\log^{10000k\log k}X, Q_1^{1/20}}$ one then has
\begin{enumerate}
\item 	(estimate for the principal character) 
\begin{multline*}
 \int_{X}^{2X} \Bigabs{\sum_{\frac{x}{q_0}<n\leq \frac{x}{q_0}+H} g(q_0 n)\chi_0(n)n^{iT_0}-\int_{\frac{x}{q_0}}^{\frac{x}{q_0}+H}u^{iT_0}\,\rd u\cdot Y^{-1}\sum_{\frac{x}{q_0}<n\leq \frac{x}{q_0}+Y}g(q_0n)\chi_0(n)}^2\,\rd x \\
		\ll P_1^{-1/8} XH^2\log^{2k-2}X,
	\end{multline*}
	\item (estimate for  non-principal characters)
	\[
	 \sum_{\chi\neq\chi_0\mmod q} \int_{X}^{2X} \Bigabs{ \sum_{x/q_0<n\leq x/q_0+H} g(q_0 n)\chi(n) n^{iT_0}}^2\,\rd x\ll P_1^{-1/8} XH^2\log^{2k-2}X.
	\]
\end{enumerate}

\end{proposition}

\begin{remark}

Proposition \ref{long-short} provides a stronger version of the short average results for divisor functions, meaning that the error term is significantly sharper than that of \cite[Theorem 1.3]{Sun}. Generally, it is difficult to achieve such a strong error term for general divisor-bounded multiplicative functions. We note, however, that \cite[Proposition 3.2]{Wang} had already established a strong exceptional set result for these functions.

\end{remark}

To prove the above proposition, we need to record the following mean value theorem for Dirichlet polynomials.

\begin{lemma}[Log-free mean value theorem] \label{mean-value}
Fix $\delta\in(0,1)$. Let $a(n)$ be a sequence with $|a(n)|\leq d_k(n)$. Let $X\geq Y\geq X^\delta\geq 2$. Then one has
\begin{multline*}
\sum_{\chi\mmod q} \int_{-T}^T\Bigabs{\sum_{X<n \leq X+Y}\frac{a(n)\chi(n)}{n^{1+it}}}^2\,\rd t \\
\ll \frac{T \varphi(q)}{X^2} \sum_{X<n\leq X+Y}|a(n)|^2 +\frac{Y}{X} \Bigbrac{\frac{\varphi(q)}{q}}^2\prod_{p\leq 2X\atop p\nmid q} \Bigbrac{1+\frac{2k-2}{p}}.	
\end{multline*}

\end{lemma}

\begin{proof}
The desired estimate holds by following the argument of \cite[Lemma 5.2]{MRTII} with $f(n)=d_k(n)$.	
\end{proof}

\noindent\emph{Proof of Proposition \ref{long-short}.}

Let $t_\chi$ denote (one of) the numbers $t\in[-X,X]$ that minimize the distance
\[
\sum_{p\leq 2X}\frac{g(p)- \Re g(p)\chi(p) p^{iT_0-it}}{p}.
\]
We plan to prove the proposition by demonstrating the claims 
 \begin{multline*}
	\sum_{\chi\mmod q} \int_X^{2X} \Bigabs{\frac{1}{H}\sum_{\frac{x}{q_0}<n\leq \frac{x}{q_0}+H} g(q_0n)\chi(n) n^{iT_0}-H^{-1}\int_{\frac{x}{q_0}}^{\frac{x}{q_0}+H} u^{it_\chi}\,\rd u \\\cdot Y^{-1}\sum_{\frac{x}{q_0}<n\leq \frac{x}{q_0}+Y} g(q_0n)\chi(n)n^{i(T_0-t_\chi)}}^2\,\rd x
		\ll P_{1}^{-1/8} X\log^{2k-2}X,
 \end{multline*}
 and 
\begin{equation*}
\sum_{\chi\neq\chi_0\mmod q}\bigabs{\sum_{x/q_0<n\leq x/q_0+Y}g(q_0 n)\chi(n) n^{i(T_0-t_\chi)}}\ll P_1^{-1/8}Y\log^{k-1}X.
\end{equation*}
 The first conclusion then follows immediately from the first claim, noting that $t_{\chi_0}=T_0$. The second conclusion is derived by combining the two claims and applying the triangle inequality. Furthermore, it is worth noting that the second claim is essentially a direct consequence of Proposition \ref{non-principal}, given that $|T_0-t_\chi|\leq 2X$ and $P_1\leq \log^{10000k\log k}X$. Therefore, it suffices to prove only the first one.

Similarly to the approach employed in the proof of \cite[Proposition 3.3]{Wang}, we set a parameter $W=\log^{10000k\log k}X$ and define the Dirichlet polynomial $G(s)$ as  $$G(s)=\sum_{n\sim X/q_0}\frac{g(q_0n)\chi(n)n^{iT_0}}{n^{s}}.$$ Applying Perron's formula to each short average we get 

\begin{align*}
\sum_{\frac{x}{q_0}<n\leq \frac{x}{q_0}+H} g(q_0n)\chi(n) n^{iT_0}=& \frac{1}{2\pi i} \int_{|t-t_\chi|\leq W}G(1+it)\frac{(\frac{x}{q_0}+H)^{1+it}-\bigbrac{\frac{x}{q_0}}^{1+it}}{1+it}\,\rd t 
\\
&+ \frac{1}{2\pi i} \int_{|t-t_\chi|> W}G(1+it)\frac{(\frac{x}{q_0}+H)^{1+it}-\bigbrac{\frac{x}{q_0}}^{1+it}}{1+it}\,\rd t  .
\end{align*}

By making the change of variables $t\to t+t_\chi$, we express the inner quotient as an integral of a phase function, namely
\begin{multline*}
\frac{H^{-1}}{2\pi i}\int_{|t-t_\chi|\leq W}G(1+it)\frac{(x/q_0+H)^{1+it}-(x/q_0)^{1+it}}{1+it}\,\rd t\\
=	\frac{H^{-1}}{2\pi i}\int_{|t|\leq W} G(1+i(t+t_\chi))\int_{x/q_0}^{x/q_0+H} u^{i(t+t_\chi)}\,\rd u\,\rd t.
\end{multline*}
Since Taylor expansion yields $u^{it}= (x/q_0)^{it}+ O(H|t|q_0/x)$ whenever $x/q_0<u\leq x/q_0+H$,  the above integral equals
\begin{align}\label{main-short-integral}
	\frac{H^{-1}}{2\pi i}\int_{x/q_0}^{x/q_0+H} u^{it_\chi}\,\rd u \cdot&  \int_{|t|\leq W} G(1+i(t+t_\chi)) \Big(\frac{x}{q_0}\Big)^{it}\,\rd t \nonumber
\\
&+O\Bigbrac{\frac{Hq_0}{x}\int_{|t|\leq W}|t||G(1+i(t+t_\chi))|\,\rd t}.
\end{align}
We shall first examine the above error term. Indeed, the application of Cauchy-Schwarz inequality permits one to deduce that
\begin{align*}
\sum_{\chi\mmod q}\int_X^{2X}& \Bigabs{ \frac{q_0 H}{x}\int_{|t|\leq W}|t||G(1+i(t+t_\chi))|\,\rd t}^2\rd x\\
&\ll \frac{q_0^2 W^3H^2}{X}\sum_{\chi\mmod q}\int_{|t|\leq W}|G(1+i(t+t_\chi))|^2\,\rd t.
\end{align*}
We then apply Lemma \ref{mean-value} for the choice $a_n=g(q_0 n)\chi(n)n^{i(T_0-t_\chi)}$, upon noting that $|a_n|\leq d_k(n)d_k(q_0)$, to bound the above expression by
\[
\ll \frac{q_0^2 W^3H^2}{X} \bigbrac{\frac{W}{X}d_k(q_0)^{2}\log^{k^2}X +\log^{2k-2}X}\ll X^{1-\eps},
\] 
where we employed the fact that $q_0,W, P_1\ll \log^{10000k\log k}X$ and $H\leq X^{1-\eps }$.

For each long average involving the character $\chi\mmod q$, another application of Perron's formula entails
\begin{multline*}
	H^{-1}\int_{x/q_0}^{x/q_0+H} u^{it_\chi}\,\rd u\cdot Y^{-1} \sum_{x/q_0<n\leq x/q_0+Y} g(q_0 n)\chi(n) n^{i(T_0-t_\chi)}\\
	=  \frac{H^{-1}}{2\pi i} \int_{\frac{x}{q_0}}^{\frac{x}{q_0}+H} u^{it_\chi}\,\rd u\cdot \Bigbrac{\int_{|t|\leq W}+\int_{|t|>W}} G(1+i(t+t_\chi))\frac{(\frac{x}{q_0}+Y)^{1+it}-(\frac{x}{q_0})^{1+it}}{Y(1+it)}\,\rd t.
\end{multline*} 
It is then worth observing that since the Taylor expansion leads to $$\frac{(\frac{x}{q_0}+Y)^{1+it}-(\frac{x}{q_0})^{1+it}}{Y(1+it)}=\Big(\frac{x}{q_0}\Big)^{it} (1+O(Y|t|q_0/x)),$$ one can deduce that
\begin{multline*}
H^{-1}\int_{\frac{x}{q_0}}^{\frac{x}{q_0}+H} u^{it_\chi}\,\rd u\cdot	\frac{1}{2\pi i} \int_{|t|\leq W} G(1+i(t+t_\chi))\frac{(\frac{x}{q_0}+Y)^{1+it}-(\frac{x}{q_0})^{1+it}}{Y(1+it)}\,\rd t\\
=H^{-1}\int_{\frac{x}{q_0}}^{\frac{x}{q_0}+H} u^{it_\chi}\,\rd u\cdot	\frac{1}{2\pi i} \int_{|t|\leq W} G(1+i(t+t_\chi)) \Big(\frac{x}{q_0}\Big)^{it}\,\rd t\\+O\Bigbrac{\frac{Y}{x/q_0}\int_{|t|\leq W}|t||G(1+i(t+t_\chi))|\,\rd t}.
\end{multline*}

In view of the previous calculation, it is apparent that the above error term is negligible. Moreover, comparing the above estimate with (\ref{main-short-integral}) one deduces that the main terms are identical. This implies that the contribution of the region $|t|\leq W$ is acceptable when evaluating the difference between the long and short averages.

In view of the above remarks, it therefore suffices to prove that whenever $H\leq h\leq Y$, the following inequality 
\begin{align*}
\sum_{\chi\mmod q}\int_X^{2X} \Bigabs{h^{-1}\int_{|t-t_\chi|>W}&G(1+it)\frac{(\frac{x}{q_0}+h)^{1+it}-(\frac{x}{q_0})^{1+it}}{1+it}\,\rd t}^2\rd x
\\
& \ll P_{1}^{-1/8} X \log^{2k-2}X.
\end{align*}
holds. We first make the change of varibles $x/q_0\to x$ and employ \cite[Lemma 8.1]{MRII} to conclude that the above left-hand side is bounded above by a constant times
\[
 \max_{T\geq X/(h q_0)} \frac{X^2/(hq_0)}{T} \sum_{\chi\mmod q} \int_{W<|t-t_\chi|\leq T} |G(1+it)|^2\rd t.
\]

We next introduce for further convenience and each $1\leq j\leq 3$ the Dirichlet polynomials 
\[
P_j(t,\chi) := \sum_{P_j\leq p\leq Q_j}\frac{k\chi(p)}{p^{1+i(t-T_{0})}}.
\] 
By taking $D=P_1^{1/6}$ we decompose each $P_j(t,\chi)$  as 
\[
P_j(t,\chi) =\sum_{\floor{D \log P_j}\leq v\leq D\log Q_j} Q_{v,j} (t,\chi) \, \text{ and }\, Q_{v,j} (t,\chi)=\twosum{e^{v/D}<p\leq e^{(v+1)/D}}{P_j\leq p\leq Q_j} \frac{k\chi(p)}{p^{1+i(t-T_{0})}}.
\]
Similarly to the argument in \cite[Section 5]{MRTII}, we now set $\delta_1=1/4-1/100$ and $\delta_2=1/4-1/50$.  Additionally, for each Dirichlet character $\chi$ modulus $q$  we write 
$$\mathcal T(\chi)=\set{t\in \R :W<|t-t_\chi|\leq T}$$
 and divide it into three disjoint ranges, namely
\begin{align}\label{pq-11}
& \mathcal T_1(\chi) = \set{ t\in\mathcal T(\chi): |Q_{v,1}(t,\chi)|\leq e^{-\delta_1 v/D} \text{ for all } \floor{D \log P_1}\leq v\leq D\log Q_1},\\
& \mathcal T_2(\chi) = \set{ t\in\mathcal T(\chi)\backslash \mathcal T_1(\chi): |Q_{v,2}(t,\chi)|\leq e^{-\delta_2 v/D} \text{ for all } \floor{D \log P_2}\leq v\leq D\log Q_2}\nonumber,\\
& \mathcal T_3(\chi) = \mathcal T(\chi)\backslash (\mathcal T_1(\chi)\cup \mathcal T_2(\chi))\nonumber.
\end{align}
We note that the function $g$ is supported on a set of natural numbers having prime factors in each of the intervals $[P_j,Q_j]$ for all $1\leq j\leq 3$. Consequently, the sum $B_5(s)$ in the proof of \cite[Lemma 4.6]{Sun} would not appear when applied in the present context.  Therefore, for each $j=1,2,3$, we use the argument of \cite[Lemma 4.6]{Sun} with Lemma \ref{mean-value} replacing \cite[Lemma 4.3]{Sun} to conclude that
\begin{multline} \label{prop3.3-reduction}
	\sum_{\chi\mmod q} \int_{\mathcal T_j(\chi)} |G(1+it)|^2\rd t \ll  P_1^{-1/6}\log^{2k-2}X\\
	+D\log Q_j  \sum_{\floor{D \log P_j}\leq v\leq D\log Q_j} \sum_{\chi\mmod q} \int_{\mathcal T_j(\chi)}|Q_{v,j(t,\chi)}R_{v,j}(t,\chi)|^2\rd t,
\end{multline}
with 
\[
R_{v,j}(t,\chi) =\sum_{\substack{(X/q_0) e^{-v/D}<m\leq (2X/q_0) e^{-v/D}\\ q_{0}m\in\mathcal{S}_{j}^{*}}} \frac{a_{k}(q_0 m)\chi(m)}{m^{1+it-iT_{0}}(\omega_{(P_j,Q_j)}(m)+1)},
\]
where $\omega_{(P_j,Q_j)}(m)$ counts the number of prime factors of $m$ in the interval $[P_j,Q_j]$ and $\mathcal{S}_{j}^{*}$ is a subset of $\supp(g)$ (that differs from $\supp(g)$ through the number of prime factors having decreased by one and the possibility of there not being prime factors in $[P_{j},Q_{j}]$). Therefore, our task is reduced to proving that  \begin{equation}\label{eq3.6}
I_j: = D\log Q_j  \sum_{v} \sum_{\chi\mmod q} \int_{\mathcal T_j(\chi)}|Q_{v,j(t,\chi)}R_{v,j}(t,\chi)|^2\rd t \ll P_1^{-1/8} \frac{Tq_0}{X/h} \log^{2k-2}X
\end{equation}
holds for all $j=1,2,3$, where $v$ runs over the range $\floor{D \log P_j}\leq v\leq D\log Q_j$. When $j=1$ we recall the definition of $\mathcal T_1(\chi)$ and note that the application of Lemma \ref{mean-value} yields 
\begin{align}\label{I1}
I_1 &\ll D\log Q_1 	\sum_{\floor{D \log P_1}\leq v\leq D\log Q_1} e^{-2\delta_1v/D} \sum_{\chi \mmod q} \int_{-T}^T |R_{v,1}(t,\chi)|^2\rd t \nonumber\\
&\ll D \log Q_1 \sum_{\floor{D \log P_1}\leq v\leq D\log Q_1} e^{-2\delta_1v/D} \Bigbrac{\frac{T\varphi(q)}{X/q_0} e^{v/D} (\log X)^{2(1+\eps')k\log k} +\log^{2k-2}X}\nonumber\\
&\ll (D\log Q_1)^2 P_1^{-2\delta_1} \Big((\log X)^{2(1+\eps')k\log k} \frac{qTQ_1}{X/q_0}+(\log X)^{2k-2}\Big)\\
&\ll P_1^{-1/8} (\log X)^{2k-2}\frac{T}{(X/q_0)/h}\nonumber,
\end{align}
the last inequality holding in view of the fact that $Q_1\leq \min(H^{2/3},P_{2}^{1/100})$ and $q\leq \log^{100k\log k}X \leq H^{1/10}$ in conjunction with the assumption $H\leq h\leq Y$.

When $j=2$ we recall the definition of $\mathcal T_2(\chi)$ to get
\[
I_2\ll qD\log Q_2 \sum_{\floor{D \log P_2}\leq v\leq D\log Q_2} e^{-2\delta_2v/D} \sup_{\chi\mmod q} \int_{-T}^T |R_{v,2}(t,\chi)|^2\rd t .
\]
One may then apply the bound for  $E_2$ from \cite[pages 26-27]{Sun}, upon noting that $\log^{80k\log k}X\leq Q_1\ll P_2^{1/100}$ and $P_1\ll Q_1^{1/20}$, to conclude that
\[
I_2\ll q (\log X)^{2k-2} \exp \Bigbrac{-\frac{\log P_2}{51}  +\frac{\log Q_1}{2}  +3\log(P_1^{1/6}\log Q_2)}	\ll q(\log X)^{2k-2} Q_1^{-4/3}.
\]

We may then deduce from the restrictions $q\ll \log^{100k\log k}X$, $Q_1\geq \log^{80k\log k} X$ and $P_1\ll Q_1^{1/20}$ the estimate
\begin{align}\label{i-2}
I_2\ll P_1^{-1/8} (\log X)^{2k-2}\ll P_1^{-1/8} (\log X)^{2k-2} \frac{T}{(X/q_0)/h}.
\end{align}

The remaining task is to handle the integral over the range $\mathcal T_3(\chi)$. To such an end we replace the latter by a sum over a one-spaced subset  $\mathcal U(\chi)\subset \mathcal T_{3}(\chi)$ and thus obtain
\[
I_3 \ll q(D\log Q_3)^2 \sup_{\chi\mmod q \atop v\in[\floor{D\log P_3}, D\log Q_3]}  \sum_{t\in \mathcal U(\chi)} |Q_{v,3}(t,\chi)R_{v,3}(t,\chi)|^2.
\]
 Noting that whenever $v\in[\floor{D\log P_3}, D\log Q_3]$ it transpires that $e^{v/D}\geq P_3,$  the latter parameter being defined in (\ref{pq-23}), the application of Lemma \ref{dirichlet-prime} for $\chi$  delivers
 \[
 Q_{v,3}(t,\chi) =\sum_{e^{v/D}\leq p\leq e^{(v+1)/D}} \frac{k\chi(p)}{p^{1+i(t-T_{0})}} \ll \exp(-c(\log X)^{1/13}).
 \]
Meanwhile one has for the principal character $\chi_0$ the identity $t_{\chi_0}=T_{0}$. Thus, whenever $t\in\mathcal U(\chi_0)$ it is apparent that $|t-T_{0}|\geq W\geq (\log X)^{10000k\log k}$. Equipped with this bound we deduce via Lemma \ref{dirichlet-prime} that
\[
Q_{v,3}(t,\chi_0) =\sum_{e^{v/D}\leq p\leq e^{(v+1)/D}} \frac{k\chi_0(p)}{p^{1+i(t-T_{0})}} \ll \frac{\log X}{|t-T_{0}|+1}\ll (\log X)^{-6000k\log k}.
\]
Combining these two estimates and recalling that $q\leq (\log X)^{100k\log k}$, $D\leq P_1^{1/6}$ and $P_1\leq (\log X)^{10000k\log k}$ we may conclude that
\[
I_3 \ll (\log X)^{-8000k\log k} \sup_{\chi\mmod q \atop v\in[\floor{D\log P_3}, D\log Q_3]} \sum_{t\in \mathcal U(\chi)} |R_{v,3}(t,\chi)|^2.
\]
We then apply \cite[Theorem 9.6]{IK} to bound the preceding discrete second moment, and thus obtain
\[
I_3 \ll (\log X)^{-7000k\log k} \sup_{\chi\mmod q \atop v\in[\floor{D\log P_3}, D\log Q_3]}  \bigbrac{1+\frac{q_{0}|\mathcal U(\chi)|T^{1/2}}{Xe^{-v/D}}}.
\]
Meanwhile, we observe that whenever $t\in\mathcal T_3(\chi)$ then there exists some $v$ with $\floor{D\log P_2}\leq v\leq \floor{D\log Q_2}$ for which $|R_{2,v}(1+it)|\geq e^{-\delta_2v/D}$. Consequently, the employment of \cite[Lemma 4.5 (ii)]{Sun} for the choice $P(1+it)=R_{2,v}(1+it)$ entails
\[
|\mathcal U(\chi)|\ll X^{1/2-1/1000}.
\]
By inserting this bound into the estimate for $I_3$ and recalling the assumption $T\geq X/(hq_{0})$ we obtain
\begin{align}\label{i-3}
I_3&\ll (\log X)^{-7000k\log k}\Big(1+X^{-1/2-1/2000}T^{1/2}\Big)\nonumber
\\
&\ll P_1^{-1/8} (\log X)^{2k-2}\Big(1+X^{-1/2-1/2000}T^{1/2}\Big(\frac{hq_{0}T}{X}\Big)^{1/2}\Big)\nonumber
\\
&\ll  P_1^{-1/8} (\log X)^{2k-2} \frac{T}{(X/q_0)/h}
\end{align}
where we used the fact that $q_{0}\leq \log^{100k\log k}X$ and $Q_{3}\ll X^{\varepsilon}.$ Combining the preceding estimates thereby yields (\ref{eq3.6}) and concludes the proof of Proposition \ref{long-short}.\qed

\section{Major arcs analysis}\label{sec4}

Unlike the usual major arc settings, it is challenging to estimate the exponential sum over $[x,x+H_{1}]$ when $H_1\asymp (\log x)^{\Phi(x)}$ because there is no asymptotic formula available for sums of $k$-fold divisor function over such short intervals. Nonetheless, in light of Proposition \ref{long-short} we can still derive an asymptotic formula for $d_k$-bounded multiplicative functions in almost all very short intervals.

In order to begin the discussion we first take $Q=\log^{100k\log k}X$, define the major arcs $\major$ to be the union of the intervals
\begin{align} \label{major-minor}
\major(q,a)= \cup_{q\leq Q} \set{\alpha\in [0,1): \lvert \alpha-a/q\rvert\leq Q^2/H_1}, 
\end{align}
and denote $\minor=\T\backslash\major$. It then seems worth introducing the expression
\begin{align}\label{exponential-fk}
	S_{f_{k}}(\alpha, x)=\sum_{x\leq n\leq x+H_{1}}f_{k}(n)e(\alpha n).
\end{align}
We further write for each $\alpha\in [0,1)$ and integers $a\in \mathbb{Z}$ and $q\in\mathbb{N}$ with $0\leq a\leq q$ and $(a,q)=1$ the frequency $\beta=\alpha-a/q$. We then introuce for $x\in \mathbb{R}$ the parameter $H_{x}=\lfloor H_{1}+x\rfloor-\lfloor x\rfloor$ and consider the auxiliary exponential sum
$$v_{x}(\beta)=e(\beta \lfloor x\rfloor)\sum_{1\leq m\leq H_{x}}e(\beta m).$$ 
\begin{lemma}\label{lem4.1}
Let $\beta\in [-1/2,1/2]$ and $x\in\R $. Then 
$$v_{x}(\beta)\ll \frac{H_{1}}{1+H_{1}\lvert \beta\rvert}.$$
\end{lemma}
\begin{proof}
When $\lvert \beta\rvert\leq H_{1}^{-1}$ the above estimate follows trivially. If instead $\lvert \beta\rvert>H_{1}^{-1}$ we employ the bound $v_{x}(\beta)\ll \lvert \beta\rvert^{-1}$.
\end{proof}

In what follows we shall utilise classical major arc manoeuvres to provide a suitable approximation for $S_{f_{k}}(\alpha, x)$. To such an end we first take $Y= X/\log^AX$ for some large constant $A>1$. We also introduce for further use the expressions
\begin{align}\label{E1k}
E_{1,k}(m)=\sum_{q_{0}q_{1}=q}\frac{\sqrt{q_{1}}}{\varphi(q_{1})}\sum_{\chi\neq \chi_{0}\mmod{q_{1}}}\Big\lvert\sum_{\substack{x/q_{0}<n\leq x/q_{0}+m/q_{0}}}f_{k}(q_{0}n)\chi(n)\Big\rvert
\end{align} 
and
\begin{align}\label{e_2}
E_{2,k}(m)=\sum_{q_{0}q_{1}=q}\frac{ \mu(q_{1})^{2}}{\varphi(q_{1})}\Big\lvert\sum_{\frac{x}{q_{0}}< n\leq \frac{x}{q_{0}}+\frac{m}{q_0}}\chi_{0}(n)f_{k}(q_{0}n)-\frac{m}{q_{0}Y}\sum_{\frac{x}{q_{0}}<n\leq \frac{x}{q_{0}}+Y}\chi_{0}(n)f_{k}(q_{0}n)\Big\rvert.	
\end{align}

\begin{lemma}\label{lem4.2}
Let $\alpha\in \T$ and $x\in\mathbb{R}$. Let $a\in \mathbb{Z}$ and $q\in\mathbb{N}$ with $0\leq a\leq q$ and $(a,q)=1$, and let $Q_1$ be a parameter such that $(\log X)^{10k\log k}\leq Q_1^6\leq H_1$. Then,
\begin{align*}&S_{f_{k}}(\alpha, x)=\frac{v_{x}(\alpha-a/q)}{Y}\sum_{q_{0}q_{1}=q}\frac{\mu(q_{1})}{q_{0}\varphi(q_{1})}\sum_{x/q_{0}< n\leq x/q_{0}+Y}\chi_{0}(n)f_{k}(q_{0}n)
\\
&+O\big(H_{1}Q_{1}^{-1}+E_{1,k}(H_{1})+ E_{2,k}(H_{1})\big)+O\Big(\big\lvert \alpha-\frac{a}{q}\big\rvert \sum_{Q_{1}^{2}\leq m\leq H_{x}}\big(E_{1,k}(m)+E_{2,k}(m)\big)\Big).
\end{align*} 
\end{lemma}
\begin{proof}

We write $\beta=\alpha-a/q$. For $0\leq m\leq H_{x}$ we consider 
$$
B(m)=\sum_{x< n\leq x+m}f_{k}(n)e(an/q),
$$
 and note that an application of summation by parts delivers 
\begin{align}\label{Sf}
S_{f_{k}}(\alpha, x) =& B(H_{x})e(\beta(\lfloor H_{1}+x\rfloor +1))+\sum_{m\leq H_{x}}B(m)e(\beta \lfloor x\rfloor)\big(e(\beta m)-e(\beta(m+1))\big).
\end{align}
We split for convenience the second sum into the ranges $[1,Q_1^2]$ and $(Q_1^2,H_x]$, and observe upon recalling (\ref{f_k}) that the sum over the first interval is negligible, namely
\begin{align}\label{eqB}
\sum_{m \leq Q_1^2} \lvert B(m)\rvert &\ll Q_1^2 \sum_{x\leq n \leq x+Q_1^2} d_k(n) 1_{\Omega(n)\leq (1+\eps')k\log \log X}\nonumber
\\
&\ll Q_1^2 k^{ (1+\eps')k\log\log  X}	 \sum_{x\leq n \leq x+Q_1^2} 1 \ll Q_1^4 (\log X)^{(1+\eps') k\log k}.
\end{align}

We next partition the summation defining $B(m)$ according to the value of $(n,q)$, observing that
\begin{align*}
B(m) &= \sum_{q_0|q} \sum_{\substack{x/q_{0}<n\leq x/q_{0}+m/q_{0} \\ (n,q/q_{0})=1}} f_k(q_0n) e\Bigbrac{\frac{an}{q/q_0}}	\\
&= \sum_{q_0|q} \twosum{r \mmod{q/q_0}}{(r,q/q_0)=1} e\Bigbrac{\frac{ar}{q/q_0}} \sum_{\substack{x/q_{0}<n\leq x/q_{0}+m/q_{0} \\ n \equiv r \mmod {q/q_0}}} f_k(q_0n).
\end{align*}
By introducing Dirichlet characters $\chi\mmod{q/q_0}$ and employing orthogonality, one has
\[
B(m) = \sum_{q_0|q} \twosum{r \mmod{q/q_0}}{(r,q/q_0)=1} e\bigbrac{\frac{ar}{q/q_0}} \varphi(q/q_0)^{-1} \sum_{\chi\mmod{q/q_0}} \bar{\chi} (r) \sum_{\frac{x}{q_{0}}<n\leq \frac{x}{q_{0}}+\frac{m}{q_{0}}} f_k(q_0n) \chi(n).
\]
We set $q_1=q/q_0$ for the sake of concission to note that then 
\begin{align}\label{eq9.1}B(m)=\widetilde{E_{1,k}}(m)+\sum_{q_{0}q_{1}=q}\frac{\mu(q_{1})}{\varphi(q_{1})}\sum_{\substack{x/q_{0}<n\leq x/q_{0}+m/q_{0}\\ (n,q_{1})=1}}f_{k}(q_{0}n),
\end{align}
wherein
$$
\widetilde{E_{1,k}}(m)=\sum_{q_{0}q_{1}=q}\frac{1}{\varphi(q_{1})}\sum_{\chi\neq \chi_{0}\mmod{q_{1}}}C_{\chi}(a,q_1)\sum_{\substack{x/q_{0}<n\leq x/q_{0}+m/q_{0}}}f_{k}(q_{0}n)\chi(n),
$$
and $$C_{\chi}(a,q_1) = \sum_{\substack{l \mmod{q_{1}}\\ (l,q_{1})=1}}e(al/q_{1})\bar\chi(l).$$ The assumption $(a,q)=1$ permits one to deduce the bound $C_{\chi}(a,q_1)\ll q_1^{1/2}$, from where it follows upon recalling (\ref{E1k}) that
\[
\widetilde{E_{1,k}}(m)\ll E_{1,k}(m).
\]

We shall next address the second term in the right side of (\ref{eq9.1}). To such an end we anticipate that the function $f_k(q_0n)\chi(n)$ in short intervals would exhibit a similar behaviour to its counterpart in long intervals. Equipped with this remark we then express the second summand in (\ref{eq9.1}) as 
$$\sum_{q_{0}q_{1}=q}\frac{\mu(q_{1})}{\varphi(q_{1})}\sum_{\substack{x/q_{0}<n\leq x/q_{0}+m/q_{0}\\ (n,q_{1})=1}}f_{k}(q_{0}n)=B_{1}(m)+O\big(E_{2,k}(m)\big),$$
where upon recalling that $Y=X\log^{-A}X$ the first term is defined by means of
$$B_{1}(m)=\sum_{q_{0}q_{1}=q}\frac{\mu(q_{1})}{\varphi(q_{1})}\frac{m}{q_{0}Y}\sum_{x/q_{0}< n\leq x/q_{0}+Y}\chi_{0}(n)f_{k}(q_{0}n),$$
and $E_{2,k}(m)$ was introduced in (\ref{e_2}).

Inserting the preceding equations and (\ref{eqB}) into (\ref{Sf}) and employing the triangle inequality and the mean value theorem it follows by the assumption $Q_1^6\leq H_1$ that
\begin{align}\label{eq6.211}&e(-\beta \lfloor x\rfloor)S_{f_{k}}(\alpha, x)=B_{1}(H_{x})e\big(\beta(H_{x}+1)\big)+\sum_{Q_{1}^{2}\leq m\leq H_{x}}B_{1}(m)\big(e(\beta m)-e(\beta(m+1))\big)\nonumber
\\
&+O\big(H_{1}Q_{1}^{-1}+ E_{1,k}(H_{x})+ E_{2,k}(H_{x})\big)+O\Big(\lvert \beta\rvert\sum_{Q_{1}^{2}\leq m\leq H_{x}}\Big( E_{1,k}(m)+ E_{2,k}(m)\Big)\Big).
\end{align}
We then observe that
\[
H_x e\big(\beta(H_{x} +1)\big)+\sum_{m\leq H_x} m\bigbrac{e(m\beta)-e((m+1)\beta)} = \sum_{m\leq H_x} e(m\beta) = v_x(\beta) e(-\beta\lfloor x\rfloor),
\]
and conclude the proof by combining the preceding line with (\ref{eq6.211}).
\end{proof}

The upcoming discussion shall be devoted to estimate the contribution of the error terms in Lemma \ref{lem4.2} on average. To such an end it seems convenient introducing beforehand some notation. Let $a\in \mathbb{Z}$ and $q\in\mathbb{N}$ with $0\leq a\leq q\leq Q$ and $(a,q)=1$. For each $x\in [X,2X]$ and $\alpha\in \major(a,q)$ we put
\begin{equation}
\label{Ups}\Upsilon_{k}(\alpha,x)=\frac{v_{x}(\alpha-a/q)}{Y}\sum_{q_{0}q_{1}=q}\frac{\mu(q_{1})}{q_{0}\varphi(q_{1})}\sum_{x/q_{0}< n\leq x/q_{0}+Y}\chi_{0}(n)f_{k}(q_{0}n).
\end{equation}
In view of the fact that the arcs $ \major(a,q)$ are disjoint, this defines a function $\Upsilon_{k}(\cdot, x):\major\to \C$. We shall also make an abuse of notation by denoting $\lvert \beta\rvert$ to the function defined in $\major$ by $\lvert \alpha-a/q\rvert$ whenever $\alpha\in\major(a,q)$.

\begin{lemma}[Error term bounds on average]\label{lem4.3}
Let $k\geq l\geq 2$ and $h\in [1,X^{1-\varepsilon}]$. Let $f_k:[X,2X]\to \R$ be as in (\ref{f_k}) and $H_1\geq(\log X)^{\Phi(X)}$, where $\Phi(X)\to \infty$ is as in Theorem \ref{thm1.1}. Then one has
\begin{align*}
\int_{X}^{2X}\Big\lvert\int_\major \Big(S_{f_{k}}(\alpha,x) \overline{S_{f_{l}}(\alpha,x)}- \Upsilon_{k}(\alpha,x)\overline{\Upsilon_{l}(\alpha,x)}\Big)e(h\alpha)\rd\alpha\Big\rvert \rd x\ll \frac{XH_1(\log X)^{k+l-2}}{Q}.
\end{align*}
\end{lemma}
\begin{proof}
We write for convenience $(k_{0},k_{1}):=(k,l)$. We employ Lemma \ref{lem4.2} for the choice $Q_1=Q_{1^{(2)}}$ as defined in (\ref{pq-11}) to deduce that the left side of the above equation is bounded above by a constant times
\begin{align}\label{eq4.4}\sum_{i\in\mathbb{Z}_{2}}&\Bigg(\int_{X}^{2X}\int_{\major}\lvert\Upsilon_{k_{i}}(\alpha,x)\rvert\Big(H_{1}Q_{1}^{-1}+\lvert \beta\rvert\sum_{Q_{1}^{2}\leq m\leq H_{x}}\big(E_{1,k_{i+1}}(m)+E_{2,k_{i+1}}(m)\big)\Big)\rd\alpha \rd x\nonumber
\\
&+\int_{X}^{2X}\int_{\major}\lvert\Upsilon_{k_{i}}(\alpha,x)\lvert\big( E_{1,k_{i+1}}(H_{x})+ E_{2,k_{i+1}}(H_{x})\big)\rd\alpha \rd x\nonumber
\\
&+\int_{X}^{2X}\int_{\major}\Big(\big(H_{1}Q_{1}^{-1}\big)^{2}+H_{1}\lvert \beta\rvert^{2}\sum_{Q_{1}^{2}\leq m\leq H_{x}}\big(E_{1,k_{i}}(m)^{2}+E_{2,k_{i}}(m)^{2}\big)\Big)d\alpha  \rd x\nonumber
\\
&+\int_{X}^{2X}\int_{\major} \big(E_{1,k_{i}}(H_{x})^{2}+ E_{2,k_{i}}(H_{x})^{2}\big)\rd\alpha \rd x\Bigg).
\end{align}
We first note by (\ref{major-minor}) that $\text{meas}(\major)\leq Q^{3}H_{1}^{-1}$ and that a direct application of Shiu's bound (\cite[Theorem 1.1]{Shiu}) in conjunction with a trivial bound for $v_{x}(\beta)$ yields
\begin{align}\label{eq4.5}\lvert \Upsilon_{k}(\alpha,x)\rvert&\ll \frac{\lvert v{_x}(\beta)\rvert}{Y}\sum_{q_{0}q_{1}=q}\frac{d_{k}(q_{0})}{q_{0}\varphi(q_{1})}\sum_{x/q_{0}< n\leq x/q_{0}+Y}d_{k}(n)\nonumber\\
&\ll H_{1} (\log X)^{k-1}\sum_{q_{0}q_{1}=q}\frac{d_{k}(q_{0})}{q_{0}\varphi(q_{1})}\ll H_{1}(\log X)^{k-1}q^{-1+\varepsilon},
\end{align}
the combination of both estimates entailing
\begin{align}\label{first-term}
H_{1}Q_{1}^{-1}\int_{\major}\lvert\Upsilon_{k}(\alpha,x)\rvert \rd\alpha\ll H_{1}Q^{3}(\log X)^{k-1}Q_{1}^{-1}\ll H_1(\log X)^{k+l-2}Q^{-2}.	
\end{align}
Note that in the preceding line we employed the fact that $Q^{5}= (\log x)^{500k\log k}$ and  $Q_{1}=(\log X)^{\Phi(X)}$. Similarly,
$$\int_{\major}\big(H_{1}Q_{1}^{-1}\big)^{2}\rd\alpha\ll H_{1}Q^{3}Q_{1}^{-2}\ll H_1(\log X)^{k+l-2}Q^{-2}.$$

We next shift the attention to the term involving $E_{1,k}(m)^{2}$ and apply Cauchy-Schwarz inequality to deduce that
\begin{align*}
&\int_{\major}\lvert \beta\rvert^{2}\sum_{Q_{1}^{2}\leq m\leq H_{x}}E_{1,k}(m)^{2}\rd\alpha  
\\
&\ll Q^{\varepsilon}\sum_{q\leq Q}\sum_{q_{0}q_{1}=q}\frac{qq_{1}}{\varphi(q_{1})}\int_{\lvert \beta\rvert\leq \frac{Q^{2}}{qH_{1}}}\lvert \beta\rvert^{2}\sum_{\substack{\chi\neq \chi_{0}\mmod{q_{1}}\\ Q_{1}^{2}\leq m\leq H_{x}}}\Big\lvert \sum_{x/q_{0}< n\leq x/q_{0}+m/q_{0}}f_{k}(q_{0}n)\chi(n)\Big\rvert^{2}\rd\beta.
\end{align*}
When $Q_{1}^{2}\leq m\leq H_{x}$ we then apply Proposition \ref{long-short} for the choices $P_{1}=P_{1^{(2)}}, Q_{1}=Q_{1^{(2)}}$, $T_0=0$ and the function $g(n)=f_k(n)$ to get 
$$
\sum_{\chi\neq\chi_0\mmod q} \int_{X}^{2X} \Bigabs{ \sum_{\frac{x}{q_{0}}<n\leq \frac{x}{q_{0}}+\frac{m}{q_{0}}} f_{k}(q_{0}n)\chi(n)}^2\,\rd x\ll P_1^{-1/8} X (mq_{0}^{-1})^{2}(\log X)^{2k-2}.
$$
Combining the preceding estimates delivers
\begin{align}\label{eq4.7}
\int_{X}^{2X}H_{1}&\int_{\major}\lvert \beta\rvert^{2}\sum_{Q_{1}^{2}\leq m\leq H_{x}}E_{1,k}(m)^{2}\rd\alpha \rd x\nonumber\\
& \ll P_1^{-1/8} X (\log X)^{2k-2} H_1 Q^\eps \sum_{Q_1^2<m \leq H_1} m^{2} \sum_{q\leq Q} \sum_{q_0q_1=q}\frac{qq_1}{q_0^2\varphi(q_{1})} \int_{\lvert \beta\rvert\leq \frac{Q^{2}}{qH_{1}}}\lvert \beta\rvert^{2} \rd\beta\nonumber\\
& \ll P_{1}^{-1/8}XH_{1}^{4}(\log X)^{2k-2} Q^{\varepsilon}\sum_{q\leq Q}q\int_{\lvert \beta\rvert\leq \frac{Q^{2}}{qH_{1}}}\lvert \beta\rvert^{2} \rd\beta\nonumber
\\
&\ll  P_{1}^{-1/8}XH_{1}(\log X)^{2k-2} Q^{6+\varepsilon},
\end{align}
as desired in view of the definitions $P_1=(\log X)^{10000k\log k}$ and $Q=(\log X)^{100k\log k}$. In an analogous manner, we note that
\begin{align*}
&H_{1}\int_{X}^{2X}\int_{\major}\lvert \beta\rvert^{2}\sum_{Q_{1}^{2}\leq m\leq H_{x}}E_{2,k}(m)^{2}\rd\alpha  \rd x
\\
&\ll Q^{\varepsilon}H_1\sum_{q\leq Q}\sum_{q_{0}q_{1}=q}\frac{q}{\varphi(q_{1})^{2}}\int_{\lvert \beta\rvert\leq \frac{Q^{2}}{qH_{1}}}\lvert \beta\rvert^{2}\rd\beta\sum_{\substack{ Q_{1}^{2}\leq m\leq 2H_{1}}}\int_{X}^{2X}\lvert G_{q_{0}}(m,x)\rvert^{2}\rd x,\nonumber
\end{align*}
where 
$$G_{q_{0}}(m,x)=\sum_{\substack{x/q_{0}< n\leq x/q_{0}+m/q_{0}\\ (n,q_{1})=1}}f_{k}(q_{0}n)-\frac{m}{q_{0}Y}\sum_{\substack{x/q_{0}< n\leq x/q_{0}+Y\\ (n,q_{1})=1}}f_{k}(q_{0}n).$$
In virtue of the inequality $m/q_{0}\geq Q_{1}^{3/2}$ whenever $Q_{1}^{2}\leq m$ as $X\to\infty$, we may now apply Proposition \ref{long-short} to the inner integral in the above line and obtain
\begin{align}\label{eq4.66}
&H_{1}\int_{X}^{2X}\int_{\major}\lvert \beta\rvert^{2}\sum_{Q_{1}^{2}\leq m\leq H_{x}}E_{2,k}(m)^{2}\rd\alpha \rd x \nonumber
\\
&\ll Q^{\varepsilon}P_1^{-1/8}XH_{1}\log^{2k-2}X\sum_{q\leq Q}\sum_{q_{0}q_{1}=q}\frac{q}{\varphi(q_{1})^{2}}\int_{\lvert \beta\rvert\leq \frac{Q^{2}}{qH_{1}}}\lvert \beta\rvert^{2}\rd\beta \sum_{\substack{ Q_{1}^{2}\leq m\leq H_{1}}}(m/q_{0})^{2}\nonumber
\\
&\ll Q^{6+\varepsilon}P_1^{-1/8}XH_{1}\log^{2k-2}X\sum_{q\leq Q}\sum_{q_{0}q_{1}=q}\frac{1}{q^{2}q_{0}^{2}\varphi(q_{1})^{2}}\ll Q^{6+\varepsilon}P_1^{-1/8}XH_{1}\log^{2k-2}X.
\end{align}
We merely indicate for the sake of concission that analogous bounds may be obtained to estimate the integrals of $E_{1,k}(H_{x})^{2}$ and $E_{2,k}(H_{x})^{2}$ on the fourth line of (\ref{eq4.4}). The preceding discussion thereby enables one to deduce that the third and fourth integrals in (\ref{eq4.4}) are $O(XH_1(\log X)^{k+l-2}Q^{-2}).$ 

In order to estimate the first and the second one we first observe that (\ref{eq4.5}) combined with the bound $\text{meas}(\major)\leq Q^{3}H_{1}^{-1}$ yields the crude estimate
$$
\int_{\major}\lvert\Upsilon_{k}(\alpha,x)\rvert^{2} \rd\alpha\ll H_{1}Q^{3}\log^{2k-2}X.
$$ Consequently, the preceding equation in conjunction with (\ref{eq4.5}) and an application of Cauchy-Schwarz combined with (\ref{eq4.7}) and (\ref{eq4.66}) delivers
\begin{align*}
&\int_{X}^{2X}\int_{\major}\lvert\Upsilon_{k}(\alpha,x)\rvert\Big(H_{1}Q_{1}^{-1}+\lvert \beta\rvert\sum_{Q_{1}^{2}\leq m\leq H_{x}}\big(E_{1,l}(m)+E_{2,l}(m)\big)\Big)\rd\alpha \rd x
\\
&+\int_{X}^{2X}\int_{\major}\lvert\Upsilon_{k}(\alpha,x)\lvert\big( E_{1,l}(H_{x})+ E_{2,l}(H_{x})\big)\rd\alpha \rd x\ll Q^{9/2}P_1^{-1/16}XH_{1}\log^{2k-2}X,
\end{align*}
which concludes the proof in view of the assumptions on $Q$ and $P_{1}$.
\end{proof}

We next present a technical lemma which shall be employed on numerous occasions in the rest of the major arc analysis. To such an end we first consider $a\in \mathbb{Z}$ and $q\in\mathbb{N}$ with $0\leq a\leq q\leq Q$ and $(a,q)=1$ and define for each $\alpha\in \major(a,q)$ the functions
 \begin{equation}\label{Falp}
 F(\alpha,x)=v_{x}(\alpha-a/q)\sum_{q_{0}q_{1}=q}\frac{\mu(q_{1})}{q_{0}\varphi(q_{1})}f_{q_{0},q_{1}}(x)
 \end{equation}
  and 
  $$G(\alpha,x)=v_{x}(\alpha-a/q)\sum_{q_{0}q_{1}=q}\frac{\mu(q_{1})}{q_{0}\varphi(q_{1})}g_{q_{0},q_{1}}(x),$$ 
  where $f_{q_{0},q_{1}}:[X,2X]\rightarrow \mathbb{R}$ and $g_{q_{0},q_{1}}:[X,2X]\rightarrow \mathbb{R}$ are given functions. The reader may find it worth noting as in (\ref{Ups}) that this defines a function in $\major$.
\begin{lemma}\label{lem4.4}
Let $C>0$ be a fixed constant and $F(x),G(x)\geq 1$ with the property that $f_{q_{0},q_{1}}(x)\ll d_{k}(q)^{C}F(x)$ and $g_{q_{0},q_{1}}(x)\ll d_{k}(q)^{C}G(x)$. Then,
\begin{equation*}
\sum_{\lvert h\rvert \leq H_{2}}\Big\lvert\int_\major F(\alpha,x) \overline{G(\alpha,x)}e(h\alpha)\rd\alpha\Big\rvert\ll  H_1H_{2}F(x)G(x).
\end{equation*}

\end{lemma} 
\begin{proof}
Let $a,q$ be as above. For each $\alpha\in \major(a,q)$,  write $\beta=\alpha-a/q$. Recalling (\ref{major-minor}) and the definitions of $F(\alpha,x)$ and $G(\alpha,x)$ we see that
\begin{align*}&\int_\major F(\alpha,x) \overline{G(\alpha,x)}e(h\alpha)\rd\alpha 
\\
&=\sum_{q\leq Q} \sum_{\substack{a=1\\ (a,q)=1}}^{q}e(ah/q)\Big(\sum_{q_{0}q_{1}=q}\frac{\mu(q_{1})}{q_{0}\varphi(q_{1})}f_{q_{0},q_{1}}(x)\Big)\Big(\sum_{q_{0}q_{1}=q}\frac{\mu(q_{1})}{q_{0}\varphi(q_{1})}g_{q_{0},q_{1}}(x)\Big)I_{q}(H_{1}),
\end{align*}
where 
\begin{equation}\label{Iq}I_{q}(H_{1})=\int_{\lvert\beta\rvert\leq \frac{Q^{2}}{qH_{1}}}\lvert v_{x}(\beta)\rvert^{2}e(\beta h)\rd\beta.\end{equation}
We shall next employ Lemma \ref{lem4.1} to estimate the above integral, and utilise the assumptions on $f_{q_0,q_1}(x)$ and $g_{q_0,q_1}(x)$ and the fact that the sum over $a$ is a Ramanujan sum to obtain
\begin{align*}\int_\major& F(\alpha,x) \overline{G(\alpha,x)}e(h\alpha)\rd\alpha
\\
&\ll \sum_{q\leq Q}\lvert c_{q}(h)\rvert \Big(\sum_{q_{0}q_{1}=q}\frac{\mu(q_{1})^{2}}{q_{0}\varphi(q_{1})}d_{k}(q)^{C}\Big)^{2} F(x)G(x) \int_{|\beta|\leq\frac{Q^2}{qH_1}} \Bigbrac{\frac{H_1}{1+H_1|\beta|}}^2\rd\beta\\
&\ll H_{1}F(x)G(x)\sum_{q\leq Q}\lvert c_{q}(h)\rvert q^{\varepsilon-2}.
\end{align*}
We pause the exposition to note that the above Ramanujan sum satisfies
$$
c_{q}(h)=\sum_{d|(q,h)}\mu\Big(\frac{q}{d}\Big)d\leq \sum_{d|(q,h)}d.
$$
Consequently, inserting the preceding inequality into the above estimate yields
\begin{align}\label{eq10}
\int_\major  &F(\alpha,x) \overline{G(\alpha,x)}e(h\alpha)\rd\alpha\ll H_{1}F(x)G(x)\sum_{q\leq Q}q^{\varepsilon-2} \sum_{d|(q,h)}d\nonumber
\\
&\ll H_{1}F(x)G(x)\sum_{d| h}\sum_{m\leq Q/d}\frac{d}{(dm)^{19/10}}\ll H_{1}F(x)G(x)\sum_{d\mid h}\frac{1}{d^{9/10}}.
\end{align}
We conclude the discussion by summing the preceding equation over $h$ to obtain
\begin{align*}\sum_{\lvert h\rvert\leq H_{2}}\Big\lvert\int_\major &F(\alpha,x) \overline{G(\alpha,x)}e(h\alpha)\rd\alpha\Big\rvert\ll H_{1}F(x)G(x)\sum_{\lvert h\rvert\leq H_{2}}\sum_{d\mid h}\frac{1}{d^{9/10}}
\\
&\ll H_{1}F(x)G(x)\sum_{d\leq H_{2}}\frac{1}{d^{9/10}}\Bigg\lfloor \frac{H_{2}}{d}\Bigg\rfloor \ll H_{1}H_{2}F(x)G(x),
\end{align*}as desired.
\end{proof}

In order to make progress in the proof it seems convenient to introduce for $ q_{0},q_{1}\in\mathbb{N}$ with $q_{0}q_{1}=q$ the meromorphic function 
$$F_{k,q_{0},q_{1}}(s)=\sum_{\substack{n=1\\ (n,q_{1})=1}}\frac{d_{k}(q_{0}n)}{n^{s}}$$ and $$P_{k,q_{0},q_{1}}(x)=\text{Res}_{s=1} \frac{x^{s}F_{k,q_{0},q_{1}}(s)}{s},$$ the latter being a polynomial of degree $k-1$ in $\log x$ (see \cite[Proposition 4.2]{MRTI}). A succinct computation employing Euler products reveals that
\begin{align}\label{eq4.20}
P_{k,q_{0},q_{1}}(x)=a_{k,q_{0},q_{1}}(\log x)^{k-1}+O\big(d_{k}(q)^{C_{k}}(\log x)^{k-2}\big),
\end{align}
where 
\begin{align*}a_{k,q_{0},q_{1}}=\frac{1}{(k-1)!}\prod_{p| q_{0} q_{1}}\Big(1-\frac{1}{p}\Big)^{k}\prod_{\substack{p| q_{0}\\ p^{i_{p}}|| q_{0}}}\Big(\sum_{j=0}^{\infty}\frac{d_{k}(p^{i_{p}+j})}{p^{j}}\Big)
\end{align*} and $C_{k}>0$ is some constant. We record for further purposes the bound
\begin{equation}\label{eqak}a_{k,q_{0},q_{1}}\leq  d_{k}(q_{0})\prod_{p| q_{0} q_{1}}\Big(1-\frac{1}{p}\Big)^{k}\prod_{p| q_{0} }\Big(1-\frac{1}{p}\Big)^{-k}\leq d_{k}(q_{0}),
\end{equation} an ensuing consequence being that
\begin{equation}\label{eqaki}P_{k,q_{0},q_{1}}(x)\ll  d_{k}(q_{0})(\log x)^{k-1}+d_{k}(q)^{C_{k}}(\log x)^{k-2}.
\end{equation}

We next recall (\ref{Ups}) and observe whenever $\alpha\in\major$ that by employing a telescoping argument one gets
\begin{align}\label{b-234}
\Upsilon_{k}(\alpha,x)=B_{2,k}(\alpha,x)+B_{3,k}(\alpha,x)+B_{4,k}(\alpha,x),
\end{align}
where the above terms are defined by means of the expressions
$$B_{2,k}(\alpha,x)=v_{x}(\beta)\sum_{q_{0}q_{1}=q}\frac{\mu(q_{1})}{q_{0}\varphi(q_{1})}P_{k,q_{0},q_{1}}( x),$$
$$B_{3,k}(\alpha,x)=\frac{v_{x}(\beta)}{Y}\sum_{q_{0}q_{1}=q}\frac{\mu(q_{1})}{q_{0}\varphi(q_{1})}\sum_{x/q_{0}< n\leq x/q_{0}+Y}\chi_{0}(n)\big(f_{k}(q_{0}n)-d_{k}(q_{0}n)\big)$$ and 
$$B_{4,k}(\alpha,x)=v_{x}(\beta)\sum_{q_{0}q_{1}=q}\frac{\mu(q_{1})}{q_{0}\varphi(q_{1})}\Big(\frac{1}{Y}\sum_{x/q_{0}< n\leq x/q_{0}+Y}\chi_{0}(n)d_{k}(q_{0}n)-P_{k,q_{0},q_{1}}( x)\Big).$$
Equipped with the preceding definitions we finally introduce the singular series
\begin{equation}\label{chk}
c_{h,k,l}=\sum_{q=1}^{\infty}c_{q}(h)\Big(\sum_{q_{0}q_{1}=q}\frac{\mu(q_{1})a_{k,q_{0},q_{1}}}{q_{0}\varphi(q_{1})}\Big)\cdot \Big(\sum_{q_{0}q_{1}=q}\frac{\mu(q_{1})a_{l,q_{0},q_{1}}}{q_{0}\varphi(q_{1})}\Big).
\end{equation}
The reader may observe that the convergence of the above series is justified by the estimate
\begin{equation*}\lvert c_{q}(h)\rvert\Big(\sum_{q_{0}q_{1}=q}\frac{\mu(q_{1})a_{k,q_{0},q_{1}}}{q_{0}\varphi(q_{1})}\Big)\cdot \Big(\sum_{q_{0}q_{1}=q}\frac{\mu(q_{1})a_{l,q_{0},q_{1}}}{q_{0}\varphi(q_{1})}\Big)\ll q^{\varepsilon-2}\sum_{d|(q,h)}d,\end{equation*} the latter stemming from (\ref{eqak}), in conjunction with the argument that leads to (\ref{eq10}).

\begin{lemma} [Major-arc analysis]\label{lem4.5}
Let $k\geq l\geq 2$. Let $f_k:[X,2X]\to \R$ be as in (\ref{f_k}) and $H_1\geq(\log X)^{\Phi(X)}$, where $\Phi(X)\to \infty$ is as in Theorem \ref{thm1.1}. Let $H_{2}\leq H_{1}^{1-\varepsilon_{1}}$ for small fixed $\varepsilon_{1}>0$. Then one has
\begin{align*}
\sum_{\lvert h\rvert \leq H_{2}}\int_{X}^{2X}\Big\lvert\int_\major S_{f_{k}}(\alpha,x) \overline{S_{f_{l}}(\alpha,x)}e(h\alpha)\rd\alpha -&c_{h,k,l}H_1(\log X)^{k+l-2}\Big\rvert \rd x
\\
&=o\big( XH_1H_{2}(\log X)^{k+l-2}\big).
\end{align*}
\end{lemma}
\begin{proof}
In view of Lemma \ref{lem4.3} it suffices to show the analogue of the preceding equation with $\Upsilon_{k}(\alpha,x)\overline{\Upsilon_{l}(\alpha,x)}$ replacing $S_{f_{k}}(\alpha,x) \overline{S_{f_{l}}(\alpha,x)}$. In light of the identity (\ref{b-234}) we observe that
\begin{align}\label{eq4.6}\int_\major &\big(\Upsilon_{k}(\alpha,x) \overline{\Upsilon_{l}(\alpha,x)}-B_{2,k}(\alpha,x)\overline{B_{2,l}(\alpha,x)}\big)e(h\alpha)\rd\alpha
\\
=& \int_\major \Big(B_{2,k}(\alpha,x) \sum_{i=3}^{4}\overline{B_{i,l}(\alpha,x)}+B_{3,k}(\alpha,x)\sum_{i=2}^{4}\overline{B_{i,l}(\alpha,x)}\Big)e(h\alpha) \rd\alpha\nonumber
\\
&+\int_{\major}B_{4,k}(\alpha,x)\sum_{i=2}^{4}\overline{B_{i,l}(\alpha,x)}e(h\alpha)\rd\alpha. \nonumber
\end{align}

We first note that a simpler variant of the proof of \cite[Lemma 3.3]{MRTII} (replacing the use of  Henriot's bound by Shiu's bound) combined with the bound $\lvert f_{k}(q_{0}n)-d_{k}(q_{0}n)\rvert\leq d_{k}(q_{0})\lvert f_{k}(n)-d_{k}(n)\rvert$, entails
\begin{equation}\label{eq13}
\sum_{x/q_{0}< n\leq x/q_{0}+Y}\chi_{0}(n)\big(f_{k}(q_{0}n)-d_{k}(q_{0}n)\big)=o\big(Y(\log X)^{k-1}\big).
\end{equation}

Moreover, the first equation at the top of page 315 of \cite{MRTII} permits one to deduce whenever $x\in [X,2X]$ and $q_{0}\leq Q$ via the fundamental theorem of calculus that 
\begin{align}\label{eq12}\sum_{x/q_{0}< n\leq x/q_{0}+Y}\chi_{0}(n)d_{k}(q_{0}n)=& \big(x/q_{0}+Y\big)P_{k,q_{0},q_{1}}\big(x/q_{0}+Y\big)-(x/q_{0})P_{k,q_{0},q_{1}}\big(x/q_{0}\big)\nonumber
\\
&+O(X\log^{-A'}X),
\end{align}
where $A'>0$ is any arbitrary large constant. One may deduce from (\ref{eq4.20}) and the bound $q_{0}\leq (\log x)^{100k\log k}$ that
$$\Big\lvert P_{k,q_{0},q_{1}}\big(x/q_{0}+Y\big)-P_{k,q_{0},q_{1}}\big(x/q_{0}\big)\Big\rvert\ll \frac{q_{0}Y}{x}(\log x)^{k-2}$$ and
$$ P_{k,q_{0},q_{1}}\big(x/q_{0}\big)=P_{k,q_{0},q_{1}}(x)+O\big((\log x)^{k-2}\log\log x\big).$$ Combining the preceding expressions and (\ref{eq12}) thereby yields
\begin{equation}\label{eq14}\sum_{x/q_{0}< n\leq x/q_{0}+Y}\chi_{0}(n)d_{k}(q_{0}n)= YP_{k,q_{0},q_{1}}(x)+O\big(Y(\log X)^{k-2}\log \log X\big).\end{equation}

We shall next prepare the ground for the application of Lemma \ref{lem4.4}. To such an end it seems first worth noting that whenever $q\leq Q$, the estimate (\ref{eq4.20}) in conjunction with (\ref{eqak}) entails \begin{equation}\label{eq15}P_{k,q_{0},q_1}(\log x)\ll (\log x)^{k-1}d_{k}(q_0).\end{equation} One then writes $B_{2,k}(\alpha,x)$, $B_{3,k}(\alpha,x)$ and $B_{4,k}(\alpha,x)$ as in (\ref{Falp}) for the choices $f_{q_0,q_1}^{(2)}(x)=P_{k,q_{0},q_{1}}( x)$,
$$
f_{q_0,q_1}^{(3)}(x)=\frac{1}{Y}\sum_{x/q_{0}< n\leq x/q_{0}+Y}\chi_{0}(n)\big(f_{k}(q_{0}n)-d_{k}(q_{0}n)\big)
$$
 and 
 $$ 
 f_{q_0,q_1}^{(4)}(x)=\frac{1}{Y}\sum_{x/q_{0}< n\leq x/q_{0}+Y}\chi_{0}(n)d_{k}(q_{0}n)-P_{k,q_{0},q_{1}}( x)
 $$respectively. We observe for further convenience that by (\ref{eq13}), (\ref{eq14}) and (\ref{eq15}) one has $$f_{q_0,q_1}^{(i)}(x)\ll d_{k}(q)^{C_{k}}F_{i}(x),\ \ \ \ i=2,3,4$$ for some large enough constant $C_{k}>0$, where $F_{2}(x)=(\log x)^{k-1}$, $F_{3}(x)=o\big((\log x)^{k-1}\big)$ and $F_{4}(x)=(\log x)^{k-2}\log\log x$. By the preceding discussion it then transpires that an application of Lemma \ref{lem4.4} delivers
$$
\sum_{\lvert h\rvert\leq H_{2}}\Big\lvert\int_\major  B_{i,k}(\alpha,x)B_{j,k}(\alpha,x) e(\alpha h) \rd \alpha\Big\rvert=o(H_{1}H_{2}(\log X)^{k+l-2})
$$
 for each $2\leq i,j\leq 4$ with $(i,j)\neq (2,2)$.
Consequently, employing the above bound to each of the summands in (\ref{eq4.6}) yields
\begin{align}\label{eq4.15}\sum_{\lvert h\rvert\leq H_{2}}\Big\lvert\int_\major \big(\Upsilon_{k}(\alpha,x) \overline{\Upsilon_{l}(\alpha,x)}-&B_{2,k}(\alpha,x)\overline{B_{2,l}(\alpha,x)}\big)e(h\alpha)\rd\alpha\Big\rvert\nonumber
\\
&=o\big(H_{1}H_{2}(\log X)^{k+l-2}\big).
\end{align} 

In order to conclude the proof it therefore suffices to evaluate the integral of $B_{2,k}(\alpha,x)\overline{B_{2,l}(\alpha,x)}$. To such an end we denote $$M_{k,q}(x)=\sum_{q_{0}q_{1}=q}\frac{\mu(q_{1})}{q_{0}\varphi(q_{1})}P_{k,q_{0},q_{1}}( x),$$ the above definition entailing the identity
\begin{align}\label{maj}&\int_{\major} B_{2,k}(\alpha,x)\overline{B_{2,l}(\alpha,x)}e(h\alpha)\rd\alpha=\sum_{q\leq Q}c_{q}(h)M_{k,q}(x)M_{l,q}(x)I_{q}(H_{1}),
\end{align} 
where $I_{q}(H_{1})$ was defined in (\ref{Iq}). In order to approximate the preceding expression by the singular series in (\ref{chk}) it is appropiate noting beforehand by using the bound (\ref{eqaki}) that
\begin{align}\label{eq4.22}&\sum_{q=1}^{\infty}q^{1/2}\lvert c_{q}(h)\rvert\Big(\sum_{q_{0}q_{1}=q}\frac{\mu(q_{1})^{2}}{q_{0}\varphi(q_{1})}\lvert P_{k,q_{0},q_{1}}( x)\rvert\Big)\Big( \sum_{q_{0}q_{1}=q}\frac{\mu(q_{1})^{2}}{q_{0}\varphi(q_{1})}\lvert P_{l,q_{0},q_{1}}( x)\rvert \Big)   \nonumber
\\
&\ll (\log x)^{k+l-2}\sum_{q=1}^{\infty}q^{\varepsilon-3/2} \sum_{d|(q,h)}d\ll (\log x)^{k+l-2}\sum_{d| h}\sum_{m=1}^{\infty}d(dm)^{-7/5}\nonumber
\\
&\ll (\log x)^{k+l-2}\sum_{d\mid h}\frac{1}{d^{2/5}}.
\end{align}
It therefore follows from the above estimate in conjunction with Lemma \ref{lem4.1} that
\begin{align*}\sum_{\lvert h\rvert\leq H_{2}}&\sum_{q\leq Q}\lvert c_{q}(h)\rvert \lvert M_{k,q}(x)M_{l,q}(x)\rvert \int_{Q^{2}/qH_{1}}^{1/2}\lvert v_x(\beta)\rvert^{2}\rd\beta
\\
&\ll Q^{-3/2}H_{1}\sum_{\lvert h\rvert\leq H_{2}}\sum_{q\leq Q}q^{1/2}\lvert c_{q}(h)\rvert \lvert M_{k,q}(x)M_{l,q}(x)\rvert
\\
&  \ll Q^{-3/2}H_{1}(\log x)^{k+l-2}\sum_{\lvert h\rvert\leq H_{2}}\sum_{d\mid h}\frac{1}{d^{2/5}}\ll Q^{-3/2}H_{1}H_{2}(\log x)^{k+l-2} .
\end{align*} 

We shall next bound the contribution of the tail of the sum in a similar manner. Indeed, the estimate (\ref{eq4.22}) combined with the aforementioned lemma yields
\begin{align*}\sum_{\lvert h\rvert\leq H_{2}}&\sum_{q> Q}\lvert c_{q}(h)\rvert \lvert M_{k,q}(x)M_{l,q}(x)\rvert \int_{-1/2}^{1/2}\lvert v_x(\beta)\rvert^{2}e(\beta h)\rd\beta
\\
&\ll Q^{-1/2}H_{1}\sum_{\lvert h\rvert\leq H_{2}}\sum_{q> Q}q^{1/2}\lvert c_{q}(h)\rvert \lvert M_{k,q}(x)M_{l,q}(x)\rvert  \ll Q^{-1/2}H_{1}H_{2}(\log x)^{k+l-2}
\end{align*} and 
\begin{align}\label{eq4.14}
\sum_{\lvert h\rvert\leq H_{2}}\sum_{q=1}^{\infty}\lvert c_{q}(h)\rvert \lvert M_{k,q}(x)M_{l,q}(x)\rvert   &\ll \sum_{\lvert h\rvert\leq H_{2}}\sum_{q=1}^{\infty}q^{1/2}\lvert c_{q}(h)\rvert \lvert M_{k,q}(x)M_{l,q}(x)\rvert   \nonumber
\\
&\ll H_{2}(\log x)^{k+l-2}.
\end{align}
Consequently, the preceding discussion and (\ref{maj}) permits one to deduce that
\begin{align}\label{eq11}&\sum_{\lvert h\rvert\leq H_{2}}\int_{\major} B_{2,k}(\alpha,x)\overline{B_{2,l}(\alpha,x)}e(h\alpha)\rd\alpha 
\\
&=\sum_{\lvert h\rvert\leq H_{2}}\sum_{q=1}^{\infty}c_{q}(h)M_{k,q}(x)M_{l,q}(x)\int_{-\frac{1}{2}}^{\frac{1}{2}}\lvert v_x(\beta)\rvert^{2}e(\beta h)\rd\beta+O\Big(Q^{-\frac{1}{2}}H_{1}H_{2}(\log x)^{k+l-2}\Big).\nonumber
\end{align}

In order to evaluate the preceding series it is worth recalling (\ref{eq4.20}), (\ref{eqak}) and (\ref{chk}) to note that
\begin{align}\label{eq4.23}\sum_{q=1}^{\infty}&c_{q}(h)M_{k,q}(x)M_{l,q}(x)=c_{h,k,l}(\log x)^{l+k-2}\nonumber
\\
&+O\Big((\log x)^{k+l-3}\sum_{q=1}^{\infty}\lvert c_{q}(h)\rvert \Big( \sum_{q_{0}q_{1}=q}\frac{\mu(q_{1})^{2}d_{k}(q)^{C_{k}}}{q_{0}\varphi(q_{1})}\Big)\Big( \sum_{q_{0}q_{1}=q}\frac{\mu(q_{1})^{2}d_{l}(q)^{C_{l}}}{q_{0}\varphi(q_{1})}\Big)\Big).
\end{align}
Employing the same argument as in (\ref{eq4.22}) we deduce that
\begin{align*}
\sum_{\lvert h\rvert\leq H_{2}}\sum_{q=1}^{\infty}\lvert c_{q}(h)\rvert& \Big( \sum_{q_{0}q_{1}=q}\frac{\mu(q_{1})^{2}d_{k}(q)^{C_{k}}}{q_{0}\varphi(q_{1})}\Big)\Big( \sum_{q_{0}q_{1}=q}\frac{\mu(q_{1})^{2}d_{l}(q)^{C_{l}}}{q_{0}\varphi(q_{1})}\Big)
\\
&\ll\sum_{\lvert h\rvert\leq H_{2}}\sum_{q=1}^{\infty} q^{\varepsilon-2}\sum_{d|(q,h)}d\ll \sum_{\lvert h\rvert\leq H_{2}}\sum_{d\mid h}\frac{1}{d^{9/10}}  \ll H_{2}.
\end{align*}
Therefore, the above estimate in conjunction with (\ref{eq4.23}) delivers
\begin{align}\label{eq4.24}\sum_{\lvert h\rvert\leq H_{2}}\Big\lvert \sum_{q=1}^{\infty}&c_{q}(h)M_{k,q}(x)M_{l,q}(x)-c_{h,k,l}(\log x)^{l+k-2}\Big\rvert\ll H_{2}(\log x)^{k+l-3}.
\end{align}

We shift the reader's attention to the integral in (\ref{eq11}) and use orthogonality to obtain whenever $\lvert h\rvert\leq H_{2}$ and $x\in \R $ that
$$
\int_{-1/2}^{1/2}\lvert v_x(\beta)\rvert^{2}e(\beta h)\rd\beta=\sum_{\substack{m_{1}-m_{2}=h\\ 0\leq m_{1}\leq H_{1}\\ 0\leq m_{2}\leq H_{1}}}1+O(1)=H_{1}+O(H_{2}).
$$
Consequently, inserting the above evaluation into (\ref{eq11}) and employing both (\ref{eq4.14}) and the fact that $H_{2}=O( H_{1}Q^{-1})$ permit one to deduce
\begin{align*}&\sum_{\lvert h\rvert\leq H_{2}}\Big\lvert\int_{\major} B_{2,k}(\alpha,x)\overline{B_{2,l}(\alpha,x)}e(h\alpha)\rd\alpha-H_{1}\sum_{q=1}^{\infty}c_{q}(h)M_{k,q}(x)M_{l,q}(x)\Big\rvert\nonumber
\\
&\ll H_{2}\sum_{\lvert h\rvert\leq H_{2}}\sum_{q=1}^{\infty}\lvert c_{q}(h)\rvert \lvert M_{k,q}(x)M_{l,q}(x)\rvert+Q^{-1/2}H_{1}H_{2}(\log x)^{k+l-2}\nonumber
\\
&\ll Q^{-1/2}H_{1}H_{2}(\log x)^{k+l-2}.
\end{align*}
Therefore, combining the previous equation with (\ref{eq4.14}) and   (\ref{eq4.24}) entails via the triangle inequality that
\begin{align*}&\sum_{\lvert h\rvert\leq H_{2}}\Big\lvert\int_{\major} B_{2,k}(\alpha,x)\overline{B_{2,l}(\alpha,x)}e(h\alpha)\rd\alpha-c_{h,k,l}H_{1}(\log x)^{l+k-2} \Big\rvert
\\
&\ll Q^{-\frac{1}{2}}H_{1}H_{2}(\log x)^{k+l-2}+H_{1}\sum_{\lvert h\rvert\leq H_{2}}\Big\lvert c_{h,k,l}(\log x)^{l+k-2}-\sum_{q=1}^{\infty}c_{q}(h)M_{k,q}(x)M_{l,q}(x)\Big\rvert
\\
&\ll   Q^{-1/2}H_{1}H_{2}(\log x)^{k+l-2}+H_{1}H_{2}(\log x)^{k+l-3}.
\end{align*}
The lemma then follows by combining the triangle inequality with Lemma \ref{lem4.3}, equation (\ref{eq4.15}) and the preceding estimate.

\end{proof}


\section{Minor arcs analysis}\label{sec5}

The minor-arc discussion in this section involves estimates for the second moment average of an exponential sum over a very narrow neighborhood centered at some frequency. Prior to \cite{MRTI}, the standard approach entailed expanding the square and using combinatorial Type I and II sums to evaluate these exponential sums. The articles \cite{MRTI, MRTII} exhibited instead a method using harmonic analysis to transfer the question to the understanding of moments of Dirichlet polynomials of the shape $\sum_{n\sim X} \frac{f_k(n)}{n^s}$. The underlying setting in the aforementioned memoir has the natural advantage that their exponential sums are supported on long intervals, making it easier to relate them to Dirichlet polynomials. However, the short exponential sums $S_{f_{k}}(\alpha, x)$ in (\ref{exponential-fk}) employed herein present additional challenges. To overcome this issue, Fourier analysis is employed. More precisely, an application of Parseval's identity  permits one to transfer the average from the dual space to the physical space, such a manoeuvre enabling one to shift the supported intervals $[x,x+H_1]$ to $[X,2X]$. We then exhibit a more flexible application of \cite[Proposition 5.1(i)]{MRTI} combined when required with the use of the Matom\"aki-Radziwi{\l}{\l} theorem.

The next lemma is an extension of \cite[Proposition 5.1 (i)]{MRTI}.

\begin{lemma}[Bounding exponential sums by mean values of Dirichlet polynomials]\label{exponential-dirichlet}
Let $1\leq H \leq  H_1/2\leq X$ be  large numbers, and let $f_k:[X,2X] \to\C$ be the function defined in (\ref{f_k}). Suppose that $1\leq a\leq q$ are coprime integers, and $\eta,\beta$ are real numbers satisfying $|\beta|\ll \eta \leq 1/2$. Then we have
\begin{multline*}
	\int_X^{2X} \int_{\beta-\frac{1}{H}}^{\beta+\frac{1}{H}} \Bigabs{ \sum_{y<n\leq y+H_1} f_k(n) e(an/q) e(n\theta)}^2\rd \theta \rd y  \\
	\ll  \frac{d_2(q)^4H_1}{q}\sup_{q=q_0q_1} \int_I \Bigbrac{\sum_{\chi \mrd {q_1}} \Bigabs{\sum_{n\sim X/q_0} \frac{f_k(q_0n) \chi(n)}{n^{1/2+it}}}}^2 \rd t \\
+\Big( H_1\Big(\eta +\frac{1}{|\beta| H}\Big)^{2}+H\Big) \int_\R \Bigbrac{H^{-1}\sum_{x<n\leq x+H} |f_k(n)|}^2\rd x,
\end{multline*}
where  
\[
I= \set{t\in \R :|\beta| (\eta X-H) \leq |t| \leq |\beta| (\eta^{-1}X+H)}.
\]

\end{lemma}

\begin{proof}
	First, we set $f(n) = f_k(n) e(an/q)$. Suppose that $\varphi:\R \to \R$ is a smooth even function supported on $[-1,1]$, equals to $1$ on the interval $[-1/10,1/10]$, and its Fourier transform $\hat \varphi (\theta) =\int_\R \varphi(y)e(-\theta y)\rd y$ obeys the bound $|\hat\varphi(\theta)| \gg 1$ for $|\theta|\leq 1$. Given any number $y\in[X,2X]$ one has
	\begin{align*}
		\int_{\beta-\frac{1}{H}}^{\beta+\frac{1}{H}} &\Bigabs{ \sum_{y<n\leq y+H_1} f_k(n) e(an/q) e(n\theta)}^2\rd \theta 
\\
&\ll \int_\R  \Bigabs{ \sum_{y<n\leq y+H_1} f(n)  e(n\theta)}^2 |\hat\varphi(H(\theta-\beta))|^2 \rd\theta
\\
&=\int_\R \Bigabs{ \int_\R \sum_{y<n\leq y+H_1} f(n) e(\beta Hw) \varphi(w) e(\theta(n-Hw))\rd w}^2 \rd\theta.	
	\end{align*}
Making the change of variables $n-Hw\to w$, one finds that the above integral equals
\[
=H^{-2} \int_\R \Bigabs{ \int_\R \sum_{y<n\leq y+H_1} f(n) \varphi\Big(\frac{n-w}{H}\Big) e(\beta (n-w)) e(\theta w)\rd w}^2 \rd\theta.
\]
It then follows from 	Plancherel's identity that the preceding line equals
\[
 H^{-2} \int_\R \Bigabs{ \sum_{y<n \leq y+H_1} f(n) \varphi\Big(\frac{n-w}{H}\Big) e(\beta n)}^2\rd w.
\]
Since the smooth function $\varphi$ vanishes outside of the interval $[-1,1]$, we may assume that the above sum runs over integers $n$ satisfying $|n-w|\leq H$. Consequently,
\begin{multline*}
\int_X^{2X} \int_{\beta-\frac{1}{H}}^{\beta+\frac{1}{H}} \Bigabs{ \sum_{y<n\leq y+H_1} f_k(n) e(an/q) e(n\theta)}^2\rd \theta \rd y \\
\ll H^{-2} \int_X^{2X} \int_\R \Bigabs{ \twosum{y<n\leq y+H_1}{w-H<n\leq w+H} f(n) \varphi\Big(\frac{n-w}{H}\Big) e(\beta n)}^2 \rd w\rd y.	
\end{multline*}	
We observe that the function inside the above integral vanishes outside of the set $w\in[y-H,y+H_1+H]$. Therefore, by switching the order of integration and recalling the assumption $H_1\geq H$ one may deduce that the preceding line is bounded above by a constant times
\begin{align}\label{eq17}
H^{-2}& \int_{X/2}^{3X} \int_{w-H_1-H}^{w+H} \Bigabs{\twosum{y<n\leq y+H_1}{w-H<n\leq w+H} f(n) \varphi\Big(\frac{n-w}{H}\Big) e(\beta n)}^2\rd y\rd w \nonumber
\\
\ll& I_{1}+I_{2}+H^{-2} \int_{X/2}^{3X} \int_{w-H_1+H}^{w-H} \Bigabs{\sum_{n} f(n) \varphi\Big(\frac{n-w}{H}\Big) e(\beta n)}^2\rd y\rd w, 
\end{align}
wherein the terms $I_{1}$ and $I_{2}$ are defined as
$$I_{1}=H^{-2} \int_{X/2}^{3X} \int_{w-H_{1}-H}^{w-H_{1}+H} \Bigabs{\sum_{w-H<n\leq w+H} f_{k}(n)}^{2} \rd y\rd w $$ and 
$$I_{2}=H^{-2} \int_{X/2}^{3X} \int_{w-H}^{w+H} \Bigabs{\sum_{w-H<n\leq w+H} f_{k}(n) }^2\rd y\rd w.$$
The reader may note that whenever $y\in [w-H_1+H,w-H]$ then one has $[w-H, w+H]\subset[y, y+H_1]$, such an remark being implicitely employed to remove the conditions on $n$ in the sum inside the integral on the third term in (\ref{eq17}). 

We first observe that one trivially has
\begin{equation*}I_{1}+I_{2}\ll H^{-1} \int_{X/2}^{3X}  \Bigabs{\sum_{w-H<n\leq w+H} f_{k}(n)}^{2} \rd w\ll H\int_{\mathbb{R}} \Bigabs{H^{-1}\sum_{w<n\leq w+H} f_{k}(n)}^{2} \rd w.\end{equation*}
The remaining of the proof shall then be devoted to estimate the third summand in (\ref{eq17}). In an analogous manner as above,
\begin{multline*}
	H^{-2} \int_{X/2}^{3X} \int_{w-H_1-H}^{w+H} \Bigabs{\sum_n f(n) \varphi\Big(\frac{n-w}{H}\Big) e(n\beta)}^2\rd y\rd w \\
	\ll \frac{H_1}{H^2} \int_{X/2}^{3X}\Bigabs{\sum_n f(n) \varphi\Big(\frac{n-w}{H}\Big)  e(n\beta)}^2\rd w.
\end{multline*}

On employing \cite[(71)]{MRTI} and recalling that $f(n)=f_k(n) e(an/q)$, one may obtain
	\begin{multline*}
\int_{X/2}^{3X}\Bigabs{\sum_n f(n) \varphi(\frac{n-w}{H}) e(n\beta)}^2\rd w
	\ll
	\frac{1}{|\beta|^{2}}\int_J \Bigbrac{ \int_{t-|\beta|H}^{t+|\beta|H} \Bigabs{\sum_{n} \frac{f_k(n) e\bigbrac{\frac{an}{q}}}{n^{1/2+it'}} } \rd t'}^{2}\rd t \\+ \Big(\eta +\frac{1}{|\beta| H}\Big)^{2}\int_\R \Bigbrac{\sum_{x<n\leq x+H} |f_k(n)|}^2\rd x,
	\end{multline*}
where $J =\set{t\in\R: \eta|\beta|X \leq |t| \leq |\beta|X/\eta}$. Next, we apply \cite[Lemma 2.9]{MRTI}, noting that $\supp(f_k)\subseteq[X,2X]$, to deduce
\[
\Bigabs{\sum_{n\sim X} \frac{f_k(n) e(an/q)}{n^{1/2+it}}} \leq \frac{d_2(q)}{q^{1/2}} \sum_{q= q_0q_1} \sum_{\chi \mrd {q_1}} \Bigabs{\sum_{n\sim X/q_0} \frac{f_k(q_0n) \chi(n)}{n^{1/2+it}}}. 
\]
It then follows from Cauchy-Schwarz inequality that
\begin{multline*}
	\int_J \Bigbrac{ \int_{t-|\beta|H}^{t+|\beta|H} \Bigabs{\sum_{n\sim X} \frac{f_k(n) e(an/q)}{n^{1/2+it'}} } \rd t'}^{2}\rd t \\
	\ll |\beta|^{2}H^{2}\frac{d_2(q)^4}{q} \sup_{q=q_0q_1} \int_I \Bigbrac{\sum_{\chi \mrd {q_1}} \Bigabs{\sum_{n\sim X/q_0} \frac{f_k(q_0n) \chi(n)}{n^{1/2+it}}}}^2 \rd t,
	\end{multline*}
	where $I =\cup_{t\in J}[t-|\beta|H,t+|\beta|H]$.
The lemma follows by combining everything together.

\end{proof}

\begin{proposition}[Minor-arc estimate]\label{prop5.2} 
Suppose that $\log^{1000k\log k}X\leq H_2\leq H_1^{1-\varepsilon_{1}}$ and $H_{1} \leq X^{1-\eps_{2}}$ for some $\eps_1,\eps_2>0$. Let $f_k:[X,2X]\to\R_{\geq0}$ be as in (\ref{f_k}) and $\minor\subset \T$ as in (\ref{major-minor}). Then we have
\begin{multline*}
\sup_{\alpha\in\T} X^{-1}\int_X^{2X} \int_{\minor\cap[\alpha-\frac{1}{2H_2},\alpha+\frac{1}{2H_2}]} \Bigabs{\sum_{y<n\leq y+H_1} f_k(n) e(n\gamma)}^2\rd\gamma \rd y\\
\ll P_{1^{(1)}}^{-1/250} H_1 \log^{2k-2}X .
\end{multline*}
\end{proposition}

\begin{proof}
Given a frequency $\alpha\in\T$, one may apply Dirichlet approximate theorem to deduce that  there are coprime integers $1\leq a\leq q\leq Q$ such that $|\alpha-a/q|\leq 1/qQ$. We write $\beta=\alpha-a/q$ and distinguish between two cases, namely $0\leq |\beta|\leq \frac{\log X}{H_2}$ and $\frac{\log X}{H_2}\leq |\beta|\leq \frac{1}{qQ}$.

\subsection*{Case 1 (for highly irrational frequencies $\gamma$)} When $\frac{\log X}{H_2}\leq |\beta|\leq \frac{1}{qQ}$, one has
\begin{multline*}
\int_X^{2X} \int_{\minor\cap[\alpha-\frac{1}{2H_2},\alpha+\frac{1}{2H_2}]} \Bigabs{\sum_{y<n\leq y+H_1} f_k(n) e(n\gamma)}^2\rd\gamma \rd y\\
\ll  \int_X^{2X} \int_{[\beta-\frac{1}{2H_2}, \beta+\frac{1}{2H_2}]} \Bigabs{\sum_{y<n\leq y+H_1} f_k(n) e(an/q) e(n\gamma)}^2\rd\gamma \rd y.
\end{multline*}

We may then apply Lemma \ref{exponential-dirichlet} with $\eta=(\log X)^{-10}$ and $H=2H_2$ to get
\begin{align}\label{beta}
	\sup_{\frac{\log X}{H_2}\leq |\beta|\leq \frac{1}{qQ}}&\int_X^{2X} \int_{[\beta-\frac{1}{2H_2}, \beta+\frac{1}{2H_2}]} \Bigabs{\sum_{y<n\leq y+H_1} f_k(n) e(an/q) e(n\gamma)}^2\rd\gamma \rd y \nonumber
\\
	\ll& \frac{H_1d_2(q)^4}{q} \sup_{q=q_0q_1} \int_{X/(2H_2 \log^{9}X)}^{2X(\log X)^{10}/qQ} \Bigbrac{\sum_{\chi \mrd {q_1}} \Bigabs{\sum_{n\sim X/q_0} \frac{f_k(q_0n) \chi(n)}{n^{1/2+it}}}}^2 \rd t  \nonumber
\\
	 &+H_1(\log X)^{-2} \int_X^{2X} \bigbrac{(2H_2)^{-1}\sum_{x<n\leq x+2H_2}f_k(n)}^2\rd x. 
\end{align}
We first square out the second term and distinguish between the diagonal and off-diagonal terms to see that
\begin{equation}\label{jj}
\sum_{x\sim X} \Bigbrac{\sum_{x<n\leq x+2H_2}f_k(n)}^2 \ll H_2\sum_{X/2\leq n\leq 3X} f_k^2(n) +H_2\sum_{0<|m|\ll H_2} \sum_{X/2\leq n\leq 3X} f_k(n) f_k(n+m). 
\end{equation}
The contribution from the diagonal terms is bounded by $O(XH_2^2(\log X)^{2k-4})$ since $f_k\leq k^{(1+\eps')k\log\log X}$ pointwise and $H_2\geq (\log X)^{1000k\log k}$, while \cite[Theorem 1.1]{MRTII} entails that the off-diagonal terms in (\ref{jj}) are $O(XH_2^{2}(\log X)^{2k-2})$, which when inserted in (\ref{beta}) delivers an error term of size $O(XH_1(\log X)^{2k-4})$. Moreover, we also notice that  the application of \cite[(27)]{MRTII} to the first term in (\ref{beta}) shows that it is $O(d_2(q)^4 P_{1^{(1)}}^{-1/8} H_1 X (\log X)^{2k-2})$, which is acceptable in view of the definition of $P_{1^{(1)}}$.

\subsection*{Case 2 (for  irrational frequencies $\gamma$)}
When instead $0\leq |\beta|\leq \frac{\log X}{H_2}$, we may focus our attention on the positive interval $[0, \frac{\log X}{H_2}]$, the analysis of its negative counterpart being analogous. We first observe that whenever $\gamma\in\minor$ is in the interval $[\alpha-\frac{1}{2H_2},\alpha+\frac{1}{2H_2}]$, the above remarks in conjunction with the triangle inequality yield
\[
\frac{Q^{2}}{qH_{1}}\leq |\gamma-a/q|\leq |\gamma-\alpha| + |\alpha-a/q| \leq  \frac{1/2+\log X}{H_{2}}.
\]
Consequently, one obtains
\begin{align*}
\int_X^{2X} &\int_{\minor\cap[\alpha-\frac{1}{2H_2},\alpha+\frac{1}{2H_2}]} \Bigabs{\sum_{y<n\leq y+H_1} f_k(n) e(n\gamma)}^2\rd\gamma \rd y\\
&\ll \int_X^{2X} \int_{\frac{Q^2}{qH_1}}^{ \frac{2\log X}{H_2}} \Bigabs{\sum_{y<n\leq y+H_1} f_k(n) e(an/q) e(n\gamma)}^2\rd\gamma \rd y.\nonumber
\end{align*}
We shall first assume $H_{1}\geq 2Q_{1^{(2)}}^{2}Q^{2}/q$ and bound the contribution of the non-empty interval $[\frac{Q^2}{qH_1},\frac{1}{Q_{1^{(2)}}^{2}}]$. We dissect such an interval by introducing the frequencies $\beta_{j}=\frac{Q^2}{qH_1}\big(1+(\log X)^{-5}\big)^{j}$ for $0\leq j\leq J$, where $J= C(\log X)^{5}\log\Big(\frac{qH_{1}}{Q^{2}Q_{1^{(2)}}^{2}}\Big)$ and $C>0$ is some suitably large constant. It then transpires that
\[
\Big[\frac{Q^2}{qH_1},\frac{1}{Q_{1^{(2)}}^{2}}\Big] \subseteq \bigcup_j \big[\beta_{j}(1-(\log X)^{-5}), \beta_j(1+(\log X)^{-5})\big],
\]
the number of such sub-intervals being $O(\log^6X)$. For each sub-interval we introduce
$$
\mathcal{I}_{j}= \set{t\in \mathbb{R}:|\beta_{j}| \big(\eta X-(\log X)^{5}/\beta_{j}\big) \leq |t| \leq |\beta_{j}|  \big(\eta^{-1}  X+(\log X)^{5}/\beta_{j}\big)}
$$
and apply Lemma \ref{exponential-dirichlet} with $\eta=(\log X)^{-10}$ and $H= \beta_{j}^{-1}(\log X)^5$ and the argument in (\ref{jj}) to see that
\begin{multline*}
\int_X^{2X} \int_{[\beta_j-\beta_j/(\log X)^5, \beta_j+\beta_j/(\log X)^5]} \Bigabs{\sum_{y<n\leq y+H_1} f_k(n) e(an/q) e(n\gamma)}^2\rd\gamma \rd y \\
\ll \frac{H_1d_2(q)^4}{q} \sup_{q=q_0q_1} \int_{\mathcal{I}_j} \Bigbrac{\sum_{\chi \mrd {q_1}} \Bigabs{\sum_{n\sim X/q_0} \frac{f_k(q_0n) \chi(n)}{n^{1/2+it}}}}^2 \rd t + XH_1(\log X)^{2k-12}.	
\end{multline*}
Summing the above line over $j$ and applying Cauchy-Schwarz then yields
\begin{align}\label{eq18}&\int_X^{2X} \int_{\big[\frac{Q^2}{qH_1},\frac{1}{Q_{1^{(2)}}^{2}}\big]} \Bigabs{\sum_{y<n\leq y+H_1} f_k(n) e(an/q) e(n\gamma)}^2\rd\gamma \rd y  
\\
& \ll H_1d_2(q)^4(\log X)^{6} \sup_{q=q_0q_1} \sum_{\chi \mrd {q_1}} \int_{\mathcal{I}} \Bigabs{\sum_{n\sim X/q_0} \frac{f_k(q_0n) \chi(n)}{n^{1/2+it}}}^2 \rd t + XH_1(\log X)^{2k-6}\nonumber,\end{align}
with
$$\mathcal{I}=\set{t\in \mathbb{R}:\frac{XQ^{2}}{2qH_{1}(\log X)^{10}}  \leq |t| \leq \frac{ X(\log X)^{11} }{Q_{1^{(2)}}^{2}} }$$ and where we employed an analogous argument to that in (\ref{beta}) to estimate the second summand.

Since in view of the assumption at the statement of the lemma one then has $H_1\leq X (\log X)^{-10000k\log k}$  and $q\leq Q\leq (\log X)^{100k\log k}$,  it suffices to prove that
\begin{align*}
d_2(q)^4 (\log X)^6 \sum_{\chi\mmod{q_1}} \int_{ (\log X)^{10000k\log k}}^{ X(\log X)^{11}/Q^{2}_{1^{(2)} } } \Bigabs{ \sum_{n\sim \frac{X}{q_0}}\frac{f_k(q_0n)\chi(n)}{n^{1/2+it}}}^2\rd t 
\ll X(\log X)^{-100k\log k}.
\end{align*}
This follows in the same spirit as the final part of the proof of Proposition \ref{long-short}, requiring only a shift of $\Re s$ from 1 to 1/2, and hence an application of \cite[Lemma 5.2]{MRTII} in lieu of Lemma \ref{mean-value} (see also the discussion after \cite[(27)]{MRTII}). Moreover, the inequalities on the definitions of the subsets $\mathcal{T}_{1}(\chi),\mathcal{T}_{2}(\chi),\mathcal{T}_{3}(\chi)$ in (\ref{pq-11}) should include a factor of $e^{v/2D}$ as in the ones right above \cite[(28)]{MRTII} . Thus, we provide a sketch of the proof here.

First, setting $P_1 =P_{1^{(2)}}, Q_1=Q_{1^{(2)}}$ and $T=X(\log X)^{11}/Q^{2}_{1^{(2)}}$ in (\ref{pq-11}), taking $q=q_1$ in (\ref{prop3.3-reduction}), and recalling that $W=(\log X)^{10000k\log k}$ in the proof of Proposition \ref{long-short}, we deduce from the appropiate version of the bound (\ref{prop3.3-reduction}) that
\[
 \sum_{\chi \mmod{q_1}} \int_{(\log X)^{10000k\log k}}^{X(\log X)^{11}/Q^{2}_{1^{(2)}}} \Bigabs{ \sum_{n\sim \frac{X}{q_0}}\frac{f_k(q_0n)\chi(n)}{n^{1/2+it}}}^2\rd t \ll  P_{1^{(2)}}^{-1/6}X (\log X)^{2k-2}+I_1+I_2+I_3.
\]
A mild modification of the estimates (\ref{i-2}) and (\ref{i-3}) demonstrate that the contributions from the terms $I_2$ and $I_3$ are acceptable. While, for $I_1$,  it follows from (\ref{I1}) for the choice of parameters  $D=P_{1^{(2)}}^{1/6}, \delta_1=1/4-1/100,$ with $q_1q_0\leq (\log X)^{100k\log k}$ and $T=X(\log X)^{11}/Q^{2}_{1^{(2)}}$ that
\begin{align*}
I_1 &\ll P_{1^{(2)}}^{-1/6+1/49} q_{1}TQ_{1^{(2)}} +XP_{1^{(2)}}^{-1/6+1/50}q_{0}^{-1}(\log X)^{2k-2}
\\
&\ll P_{1^{(2)}}^{-1/6+1/50} X (\log X)^{2k-2}.\end{align*}
Recalling the definition of $P_{1^{(2)}}$ in (\ref{pq-1}), we then conclude that the estimate for the integral over the interval $[\frac{Q^2}{qH_1},\frac{1}{Q_{1^{(2)}}^{2}}]$ is acceptable.

If instead $H_{1}<2Q_{1^{(2)}}^{2}Q^{2}/q$, it in turn entailing when applied in conjunction with the assumption $H_{2}\leq H_{1}^{1-\eps }$ that $H_{2}< Q_{1^{(2)}}^{2}\log X$, the integral range $[\frac{1}{Q_{1^{(2)}}^{2}},\frac{2\log X}{H_{2}}]$ is not empty. Under such circumstances we dissect $[\frac{1}{Q_{1^{(2)}}^{2}},\frac{2\log X}{H_{2}}]$ and introduce $\rho_{j}=\frac{1}{Q_{1^{(2)}}^{2}}\big(1+P_{1^{(1)}}^{-1/200}\big)^{j}$ for $0\leq j\leq J'$, where $J'$ is some natural number satisfying
$$
J'= CP_{1^{(1)}}^{1/200}\log\Big(\frac{2Q_{1^{(2)}}^{2}\log X}{H_{2}}\Big) \ll P_{1^{(1)}}^{3/500}
$$ 
and $C>0$ is some suitably large constant. It then transpires that
\[
\Big[\frac{1}{Q_{1^{(2)}}^{2}},\frac{2\log X}{H_{2}}\Big] \subset \bigcup_{j\leq J'} \big[\rho_{j}(1-P_{1^{(1)}}^{-1/200}), \rho_j(1+P_{1^{(1)}}^{-1/200})\big].
\]
For each sub-interval we introduce
$$
\mathcal{I}_{j}'= \set{t\in \mathbb{R}:|\rho_{j}|  \big(\eta'X-P_{1^{(1)}}^{1/200}/\rho_{j}\big) \leq |t| \leq |\rho_{j}| \big(\eta'^{-1}X+P_{1^{(1)}}^{1/200}/\rho_{j}\big)}
$$
and apply Lemma \ref{exponential-dirichlet} with $\eta'=P_{1^{(1)}}^{-1/100}$ and $H= P_{1^{(1)}}^{1/200}/\rho_j$ to see that
\begin{multline*}
\int_X^{2X} \int_{[\rho_{j}(1-P_{1^{(1)}}^{-1/200}), \rho_j(1+P_{1^{(1)}}^{-1/200})\big]} \Bigabs{\sum_{y<n\leq y+H_1} f_k(n) e(an/q) e(n\gamma)}^2\rd\gamma \rd y\\
\ll \frac{H_1d_2(q)^4}{q} \sup_{q=q_0q_1} \int_{\mathcal{I}_{j}'} \Bigbrac{\sum_{\chi \mrd {q_1}} \Bigabs{\sum_{n\sim X/q_0} \frac{f_k(q_0n) \chi(n)}{n^{1/2+it}}}}^2 \rd t  + XH_1(\log X)^{2k-2}P_{1^{(1)}}^{-1/100}.		
\end{multline*}
Summing the above line over $j\leq J'$ and applying Cauchy-Schwarz then yields
\begin{align}\label{eq18}&\int_X^{2X} \int_{\big[\frac{1}{Q_{1^{(2)}}^{2}},\frac{2\log X}{H_{2}}\big]} \Bigabs{\sum_{y<n\leq y+H_1} f_k(n) e(an/q) e(n\gamma)}^2\rd\gamma \rd y  \nonumber
\\
 \ll& H_1d_2(q)^4P_{1^{(1)}}^{3/500} \sup_{q=q_0q_1} \sum_{\chi \mrd {q_1}} \int_{\mathcal{I}'} \Bigabs{\sum_{n\sim X/q_0} \frac{f_k(q_0n) \chi(n)}{n^{1/2+it}}}^2 \rd t \\
 &+ XH_1(\log X)^{2k-2}P_{1^{(1)}}^{-1/250},\nonumber\end{align}
where
$$\mathcal{I}'=\set{t\in \mathbb{R}:\frac{X}{2Q_{1^{(2)}}^{2}P_{1^{(1)}}^{1/100}}  \leq |t| \leq \frac{ 4P_{1^{(1)}}^{1/100} X\log X}{H_{2}} }.$$ 

Choosing $P_{1}=P_{1^{(1)}}$ and $Q_{1}=Q_{1^{(1)}}$ in (\ref{pq-11}), a modification as above of (\ref{prop3.3-reduction}) (see also \cite[(27)]{MRTII}), together with the condition $H_2\geq (\log X)^{1000k\log k}$, yields 
\[
\sup_{q=q_0q_1}\sum_{\chi\mmod{q_1}}\int_{X(\log X)^{-\Phi(X)}}^{X(\log X)^{10}/H_2} \Bigabs{\sum_{n \sim X/q_0}\frac{f_k(q_0n)\chi(n)}{n^{1/2+it}}}^2\rd t  \ll P_{1^{(1)}}^{-1/8} X (\log X)^{2k-2}.
\]
The contribution from the first term in (\ref{eq18}) is thereby acceptable.  Consequently, combining the estimates for all cases concludes the proof of the lemma.
\end{proof}


\section{The large values estimates} \label{sec6}

In this section, we aim to establish a mean value estimate for the divisor function without substantial logarithmic power losings. Firstly, we notice that a direct application of Parseval's identity introduces some logarithmic loss, namely
\begin{align*}
\int_X^{2X} \int_\T \bigabs{\sum_{x<n\leq x+H_1} f_k(n) e(n\alpha)}^2\rd \alpha\rd x &\leq \int_X^{2X} \sum_{x<n\leq x+H_1} f_k^2(n)\rd x\\
 &\ll XH_1(\log X)^{(1+\eps')k\log k+k-1},	
\end{align*}
which is not acceptable due to the weak saving obtained in the minor arc analysis (see Proposition \ref{prop5.2}).

 \begin{proposition}[Large values estimates]\label{prop6.1}
Let $\eps>0$ and $J\leq H_{2}$. Assume that $\log^{1000k\log k}X\leq H_2\leq H_1 (\log X)^{-20000 k\log k}$ and $H_1\leq X^{1-\eps} $. Let $\alpha_1,\dots,\alpha_J$ be a $1/H_2$-separated sequence of $\T$ and let $f_k:[X,2X]\to\R_{\geq 0}$ be as defined in (\ref{f_k}). Then there exists a function $\psi_{1}(X)$ not depending on $\Psi(X)$ such that $\psi_{1}(X)\rightarrow 0$ with $X\rightarrow \infty$ and
\begin{align*}
\sum_{1\leq j\leq J} X^{-1}\int_X^{2X} &\int_{[\alpha_j-\frac{1}{2H_2},\alpha_j+\frac{1}{2H_2}]} \Bigabs{ \sum_{x<n\leq x+H_1} f_k(n) e(n\alpha)}^2\rd \alpha \rd x
\\
&\ll J^{1/4} H_1(\log X)^{2k-2+\psi_{1}(X)}.
\end{align*}

 \end{proposition}
 
\begin{proof}
 	
 As a prelude to the discussion, one may make the change of variables $\alpha=\alpha_j+\theta$ and apply Gallagher's lemma (\cite[Lemma 1]{gal}) to see that
\begin{align}\label{pp}
&X^{-1}\int_X^{2X} \int_{[\alpha_j-\frac{1}{2H_2},\alpha_j+\frac{1}{2H_2}]} \Bigabs{ \sum_{x<n\leq x+H_1} f_k(n) e(n\alpha)}^2\rd \alpha \rd x\nonumber\\
&\ll X^{-1} \int_X^{2X} \int_x^{x+H_1-H_{2}} \Bigabs{H_2^{-1}\sum_{y<n\leq y+H_2} f_k(n) e(n\alpha_j) }^2\rd y \rd x	+I_{2}(j)+I_{3}(j),\nonumber
\\
&\ll H_1/X \int_{X/2}^{3X} \Bigabs{H_2^{-1}\sum_{x<n\leq x+H_2} f_k(n) e(n\alpha_j) }^2\rd x+I_{2}(j)+I_{3}(j),
\end{align} 
where
$$I_{2}(j)= X^{-1} \int_X^{2X} \int_{x-H_{2}}^{x} \Bigabs{H_2^{-1}\sum_{x<n\leq y+H_2} f_k(n) e(n\alpha_j) }^2\rd y \rd x$$ and $$I_{3}(j)= X^{-1} \int_X^{2X} \int_{x+H_1-H_{2}}^{x+H_1} \Bigabs{H_2^{-1}\sum_{y<n\leq x+H_1} f_k(n) e(n\alpha_j) }^2\rd y \rd x.$$ 

We pause momentarily the discussion and note that whenever $H\geq (\log X)^{10000k\log k}$ it transpires by the third formula on \cite[page 830]{MRTII} that
$$ \sum_{1\leq j\leq J} \Bigbrac{\int_{X}^{2X} \bigabs{H^{-1}\sum_{x<n\leq x+H} f_k(n) e(n\alpha_j) }^2\rd x}^{1/2}\ll J^{5/8} X^{1/2} (\log X)^{k-1+\psi_1(X)/2}.$$ It therefore follows for each $1\leq j\leq J$ that the $j$-th largest value of
\begin{equation}\label{po}\int_{X/2}^{3X} \bigabs{H^{-1}\sum_{x<n\leq x+H} f_k(n) e(n\alpha_j) }^2\rd x\end{equation} is $O(j^{-3/4}X(\log X)^{2k-2+\psi_1(X)}).$

We set $L(X)=(\log X)^{10000k\log k}$. We then sum over $j\leq J$ the term $I_{2}(j)$, make the change of variables $h=y+H_{2}-x$  and employ the pointwise bound $f_k\leq k^{(1+\eps')k\log\log X}$ and the preceding conclusion to get
\begin{align}\label{poi}
&\sum_{1\leq j\leq J}I_{2}(j)\ll \int_{0}^{H_{2}}\sum_{1\leq j\leq J}X^{-1}\int_{X}^{2X}\Bigabs{H_2^{-1}\sum_{x<n\leq x+h} f_k(n) e(n\alpha_j) }^2\rd x\rd h \nonumber
\\
\ll & \int_{L(X)}^{H_{2}} \sum_{1\leq j\leq J}j^{-3/4}(\log X)^{2k-2+\psi_1(X)}\rd h\nonumber
\\
&+ J\int_{0}^{L(X)}X^{-1}\int_{X}^{2X}\Bigabs{H_2^{-1}\sum_{x<n\leq x+h} f_k(n) }^2\rd x\rd h\nonumber
\\
\ll & J^{1/4}H_{2}(\log X)^{2k-2+\psi_1(X)}+H_{2}(\log X)^{20000k\log k},
\end{align}
where in the last step we used the assumption $J\leq H_{2}$, an analogous argument delivering
\begin{equation}\label{poi1}\sum_{1\leq j\leq J}I_{3}(j)\ll J^{1/4}H_{2}(\log X)^{2k-2+\psi_1(X)}+H_{2}(\log X)^{20000k\log k}.\end{equation}

We next estimate the contribution from the first summand in (\ref{pp}) by employing the conclusion in (\ref{po}) as above, namely,
\begin{align*}\frac{H_{1}}{X}\sum_{1\leq j\leq J} \int_{X/2}^{3X} \Bigabs{H_2^{-1}\sum_{x<n\leq x+H_2} f_k(n) e(n\alpha_j) }^2\rd x&\ll H_{1}\sum_{1\leq j\leq J}j^{-3/4}(\log X)^{2k-2+\psi_1(X)}
\\
&\ll J^{1/4}H_{1}(\log X)^{2k-2+\psi_1(X)}.
\end{align*}
Recalling that $H_1\geq H_2(\log X)^{20000k\log k}$ and combining (\ref{pp}) with the bounds (\ref{poi}), (\ref{poi1}) and the preceding line thereby concludes the proof.	
 \end{proof}

 
\section{Circle method: Completing the proof of Proposition \ref{reduction}}\label{sec7}
In light of the definition of the major and minor arcs in (\ref{major-minor}) and the exponential sum notation from (\ref{exponential-fk}), the customary orthogonality relation reveals that
\[
\sum_{x< n\leq x+H_1} f_k(n) f_l(n+h) = \Bigset{ \int_\major +\int_\minor} S_{f_{k}}(\alpha,x) \overline{S_{f_{l}}(\alpha,x)} e(h\alpha)\,\rd\alpha.
\]
Thus, by Lemma \ref{lem4.5}, in order to prove Proposition \ref{reduction} it suffices to show that
\[
\sum_{|h|\leq H_2} X^{-1} \int_X^{2X} \Bigabs{ \int_\minor S_{f_{k}}(\alpha,x) \overline{S_{f_{l}}(\alpha,x)} e(h\alpha)\,\rd\alpha} \rd x =o(H_1H_2 (\log X)^{k+l-2}).
\]
We define for further convenience the function $c:[X,2X]\to\mathbb \mathbb{S}^{1}$ by means of the expression
$$  \Big\lvert\int_\minor S_{f_{k}}(\alpha,x) \overline{S_{f_{l}}(\alpha,x)} e(h\alpha)\,\rd\alpha\Big\rvert =c(x)\int_\minor S_{f_{k}}(\alpha,x) \overline{S_{f_{l}}(\alpha,x)} e(h\alpha)\,\rd\alpha. $$
Equipped with the above notation it therefore transpires by Cauchy-Schwarz inequality that it is sufficient to prove 
\begin{align}\label{minor-goal}
\sum_{|h|\leq H_2} \Bigabs{X^{-1} \int_X^{2X} c(x) \int_\minor S_{f_{k}}(\alpha,x) \overline{S_{f_{l}}(\alpha,x)} e(h\alpha)\,\rd\alpha \rd x}^2 =o(H_2H_1^2 (\log X)^{2k+2l-4}).
\end{align}

To such an end we introduce for simplicity an even non-negative Schwartz function $\Phi:\mathbb{R}\rightarrow \mathbb{R}^{+}$ with $\Phi(x)\geq 1$ for all $x\in [-1,1]$ such that the Fourier transform 
$$
\hat{\Phi}(\xi)=\int_{\mathbb{R}}\Phi(x)e(-x\xi)\rd x
$$ is supported on $[-1/2,1/2]$. Then, 
\begin{align*}
\sum_{|h|\leq H_2}& \Bigabs{X^{-1}\int_X^{2X}c(x)\int_\minor S_{f_{k}}(\alpha,x) \overline{S_{f_{l}}(\alpha,x)} e(h\alpha)\,\rd\alpha\rd x}^{2}
\\
&\ll \sum_{h} \Phi(h/H_{2})\Bigabs{X^{-1}\int_X^{2X} c(x) \int_\minor S_{f_{k}}(\alpha,x) \overline{S_{f_{l}}(\alpha,x)} e(h\alpha)\,\rd\alpha \rd x}^{2}.
\end{align*}

By squaring out and switching the order of integration one may find that the above expression equals 
\begin{multline*}
X^{-2}\sum_{h} \Phi(h/H_{2}) \int_X^{2X} c(x) \int_\minor  S_{f_{k}}(\alpha,x) \overline{S_{f_{l}}(\alpha,x)}    e(h\alpha)\,\rd\alpha\,\rd x \\
\cdot \int_X^{2X} \bar{c(y)}\int_\minor \overline{S_{f_{k}}(\gamma,y)}S_{f_{l}}(\gamma,y)  e(-h\gamma)\,\rd\gamma\,\rd y\\
 \ll  X^{-2}\int_X^{2X}\int_{\minor}\lvert S_{f_{k}}(\alpha,x)\rvert \lvert S_{f_{l}}(\alpha,x)\rvert \int_X^{2X}\int_{\minor}\lvert S_{f_{k}}(\gamma,y)\rvert \lvert S_{f_{l}}(\gamma,y)\rvert \\
 \cdot \Bigabs{\sum_{h} \Phi(h/H_{2}) e(h(\alpha-\gamma))} \, \rd\gamma \rd y  \rd\alpha \rd x
\end{multline*}
We employ Poisson summation on the inner sum (see the proof of \cite[Proposition 3.1]{MRTI}) to deduce that 
\begin{align*}
&\sum_{|h|\leq H_2} \Bigabs{\int_X^{2X} c(x)\int_\minor S_{f_{k}}(\alpha,x) \overline{S_{f_{l}}(\alpha,x)} e(h\alpha)\,\rd\alpha\rd x}^{2}
\\
&\ll H_{2}\int_{\minor}\int_X^{2X}\lvert S_{f_{k}}(\alpha,x)\rvert \lvert S_{f_{l}}(\alpha,x)\rvert \int_{\minor\cap [\alpha-\frac{1}{2H_{2}},\alpha+\frac{1}{2H_{2}}] }\int_X^{2X}\lvert S_{f_{k}}(\gamma,y)\rvert \lvert S_{f_{l}}(\gamma,y)\rvert\rd \bf{x} .
\end{align*}
where we wrote $\rd \mathbf{x}=\rd x\rd y \rd\gamma \rd\alpha.$ We then decompose the frequencies $\alpha\in\T$ into intervals of length $\frac{1}{H_2}$ and thus write $I_i=[\frac{i-1}{H_2},\frac{i}{H_2}]$ with $1\leq i\leq H_2$. By noting that when $\alpha\in I_i$, the frequencies $\gamma\in [\alpha-\frac{1}{2H_2},\alpha+\frac{1}{2H_2}]$ must be in $I_{i-1}\cup I_i\cup I_{i+1}$ it follows that
\begin{multline*}
\sum_{|h|\leq H_2} \Bigabs{X^{-1}\int_X^{2X} c(x)\int_\minor S_{f_{k}}(\alpha,x) \overline{S_{f_{l}}(\alpha,x)} e(h\alpha)\,\rd\alpha\rd x}^{2}\\
\ll H_2 \sum_{i=1}^{H_{2}} \Bigbrac{ \int_{\minor\cap I_i} X^{-1} \int_X^{2X} \bigabs{ S_{f_{k}}(\alpha,x) \overline{S_{f_{l}}(\alpha,x)}} \rd x \rd\alpha}^2.
\end{multline*}
Consequently, an application of Cauchy-Schwarz inequality with respect to both $x$ and $\alpha$ yields
\begin{equation}\label{eq4.25}
\sum_{|h|\leq H_2} \Bigabs{X^{-1}\int_X^{2X}c(x)\int_\minor S_{f_{k}}(\alpha,x) \overline{S_{f_{l}}(\alpha,x)} e(h\alpha)\,\rd\alpha\rd x}^{2}\ll H_2 \sum_{i=1}^{H_{2}}a(I_{i})b(I_{i}),
\end{equation}
where
$$a(I_{i})=X^{-1}\int_{X}^{2X}\int_{\minor\cap I_i}\lvert S_{f_{k}}(\alpha,x)\rvert^{2}\rd\alpha \rd x,\  \        b(I_{i})=X^{-1}\int_{X}^{2X}\int_{\minor\cap I_i}\lvert S_{f_{l}}(\alpha,x)\rvert^{2}\rd\alpha \rd x.$$ 
In order to estimate the right side of (\ref{eq4.25}) we employ Holder's inequality to get
\begin{equation}\label{eq4.26}
H_2 \sum_{i=1}^{H_{2}}a(I_{i})b(I_{i})\ll H_2 \big(\sup_{i}\lvert a(I_{i})\rvert\big)^{1/4}|| a||_{l^{7/4}}^{3/4}\cdot || b||_{l^{7/4}}.
\end{equation}
 Recalling the definition $P_{1^{(1)}}= (\log X)^{\Psi(X)}$ in (\ref{pq-1}), we first utilise Proposition \ref{prop5.2} to deduce that
$$\sup_{i}\lvert a(I_{i})\rvert\ll (\log X)^{-\Psi(X)/250} H_1 \log^{2k-2}X.$$
Moreover, Proposition \ref{prop6.1} delivers the bound
$$a(I_{i_{1}})+\cdots+a(I_{i_{J}})\ll J^{1/4} H_1(\log X)^{2k-2+\psi_{1}(X)}$$ for any collection $I_{i_{1}},\ldots, I_{i_{J}}$ of distinct intervals .This implies that the $j$-th largest value of $a(I_j)$ is $O(j^{-3/4 }H_1(\log X)^{2k-2+\psi_{1}(X)})$, the same argument delivering an analogous bound for the $j$-th largest value of $b(I_j)$. By the preceding discussion, it transpires that
$$|| a||_{l^{7/4}}\ll \Bigbrac{\sum_{j\leq H_2}\bigbrac{j^{-3/4 }H_1(\log X)^{2k-2+\psi_{1}(X)}}^{7/4}}^{4/7} \ll   H_1(\log X)^{2k-2+\psi_{1}(X)}$$ and
$$|| b||_{l^{7/4}}\ll  H_1(\log X)^{2l-2+\psi_{1}(X)}.$$
 
We then insert the previous bounds into (\ref{eq4.26}) to derive
\begin{equation*}
H_2 \sum_{i=1}^{H_{2}}a(I_{i})b(I_{i})\ll H_{1}^{2}H_{2}(\log X)^{2k+2l-4+2\psi_{1}(X)-\frac{1}{1000}\Psi(X)},
\end{equation*}
Combining the preceding line with (\ref{eq4.25}) and choosing $\Psi(X)\geq 2001 \psi_{1}(X)$ yields inequality (\ref{minor-goal}) and completes the proof of Proposition \ref{reduction}, as desired.

\bibliographystyle{plain} 

\renewcommand{\bibname}{} 

\bibliography{local_divisor_correlation.bib}

\begin{thebibliography}{10}

\bibitem{BBMZ}
S.~Baier, T.~D. Browning, G.~Marasingha, and L.~Zhao.
\newblock Averages of shifted convolutions of {$d_3(n)$}.
\newblock {\em Proc. Edinb. Math. Soc. (2)}, 55(3):551--576, 2012.

\bibitem{CG}
J.~B. Conrey and S.~M. Gonek.
\newblock High moments of the {R}iemann zeta-function.
\newblock {\em Duke Math. J.}, 107(3):577--604, 2001.

\bibitem{DI}
J.-M. Deshouillers and H.~Iwaniec.
\newblock Kloosterman sums and {F}ourier coefficients of cusp forms.
\newblock {\em Invent. Math.}, 70(2):219--288, 1982/83.

\bibitem{Dr}
S.~Drappeau.
\newblock Sums of {K}loosterman sums in arithmetic progressions, and the error term in the dispersion method.
\newblock {\em Proc. Lond. Math. Soc. (3)}, 114(4):684--732, 2017.

\bibitem{gal}
P.~X. Gallagher.
\newblock A large sieve density estimate near {$\sigma =1$}.
\newblock {\em Invent. Math.}, 11:329--339, 1970.

\bibitem{HB}
D.~R. Heath-Brown.
\newblock Prime numbers in short intervals and a generalized {V}aughan identity.
\newblock {\em Canadian J. Math.}, 34(6):1365--1377, 1982.

\bibitem{IK}
H.~Iwaniec and E.~Kowalski.
\newblock {\em Analytic number theory}, volume~53.
\newblock American Mathematical Soc., 2021.

\bibitem{Kh}
T.~Khale.
\newblock An explicit {V}inogradov-{K}orobov zero-free region for {D}irichlet {$L$}-functions.
\newblock {\em Q. J. Math.}, 75(1):299--332, 2024.

\bibitem{MRII}
K.~Matom{\"a}ki and M.~Radziwi{\l}{\l}.
\newblock Multiplicative functions in short intervals {II}.
\newblock {\em arXiv preprint arXiv:2007.04290}, 2020.

\bibitem{MRTI}
K.~Matom\"aki, M.~Radziwi\l\l, and T.~Tao.
\newblock Correlations of the von {M}angoldt and higher divisor functions {I}. {L}ong shift ranges.
\newblock {\em Proc. Lond. Math. Soc. (3)}, 118(2):284--350, 2019.

\bibitem{MRTII}
K.~Matom\"aki, M.~Radziwi\l\l, and T.~Tao.
\newblock Correlations of the von {M}angoldt and higher divisor functions {II}: divisor correlations in short ranges.
\newblock {\em Math. Ann.}, 374(1-2):793--840, 2019.

\bibitem{MV}
H.~Montgomery and R.~Vaughan.
\newblock {\em Multiplicative number theory {I}: {C}lassical theory}.
\newblock Number~97. Cambridge university press, 2007.

\bibitem{Shiu}
P.~Shiu.
\newblock A {B}run-{T}itchmarsh theorem for multiplicative functions.
\newblock {\em J. Reine Angew. Math.}, 313:161--170, 1980.

\bibitem{Sun}
Y.~Sun.
\newblock On divisor bounded multiplicative functions in short intervals.
\newblock {\em arXiv preprint arXiv:2401.08432}, 2024.

\bibitem{To}
B.~Topacogullari.
\newblock The shifted convolution of generalized divisor functions.
\newblock {\em Int. Math. Res. Not. IMRN}, (24):7681--7724, 2018.

\bibitem{Wang}
M.~Wang.
\newblock Local {F}ourier uniformity of higher divisor functions on average.
\newblock {\em arXiv preprint arXiv:2402.18342}, 2024.

\end{thebibliography}

\end{document}